\documentclass[11pt]{article}
\usepackage{amsfonts}
\usepackage{amsmath}
\usepackage{amssymb,amsthm,amsbsy,hyperref,mathrsfs,color}
\usepackage{graphicx,float}

\numberwithin{equation}{section}
\newcommand{\intt}{\int_{\mathbb{R}^{3}}}
\newtheorem{thm}{Theorem}[section]
\newtheorem{lem}{Lemma}[section]

\newtheorem{prop}{Proposition}[section]
\newtheorem{rem}{Remark}[section]

\newcommand{\beq}{\begin{eqnarray}}
\newcommand{\eeq}{\end{eqnarray}}
\newcommand{\beqno}{\begin{eqnarray*}}
\newcommand{\eeqno}{\end{eqnarray*}}
\newcommand{\ba}{\begin{align}}
\newcommand{\ea}{\end{align}}
\newcommand{\ee}{\end{equation}}

\theoremstyle{definition}

\newcommand{\non}{\nonumber}
\newcommand{\mdiv}{\mathrm{div}}
\newcommand{\D}{\displaystyle}
\newcommand{\bi}{\Big}

\newcommand{\rf}{\frac{\mathrm{d}}{\mathrm{d}t}}

\def\div{\mathop{\rm div}\nolimits}

\allowdisplaybreaks

\begin{document}
\title{\bf Global classical solution to 3D compressible magnetohydrodynamic equations with large initial data and vacuum}
\author{
Guangyi Hong\thanks{School of Mathematics and Statistics, Central
China Normal University, Wuhan 430079, China. E-mail:
hongguangyi1989@126.com},\ \ Xiaofeng Hou\thanks{School of
Mathematics and Statistics, Central China Normal University, Wuhan
430079, China. E-mail: xiaofengh0513@163.com},\ \ Hongyun
Peng\thanks{School of Mathematics and Statistics, Central China
Normal University, Wuhan 430079, China. E-mail: penghy010@163.com},\
\ Changjiang Zhu\thanks{Corresponding author. School of Mathematics, South China University of Technology, Guangzhou,
{510641}, China. E-mail: cjzhu@mail.ccnu.edu.cn}
 \\
}
\date{}
\maketitle
\begin{abstract}
In this paper, we study the Cauchy problem of the isentropic
 compressible magnetohydrodynamic equations in $\mathbb{R}^{3}$. When
  $(\gamma-1)^{\frac{1}{6}}E_{0}^{\frac{1}{2}}$, together with the $\|H_{0}\|_{L^{2}}$, is suitably small,
   a result on the existence of global classical solutions is obtained.
    It should be pointed out that the initial energy $E_{0}$  except the $L^{2}$- norm of $H_{0}$ can
    be large as $\gamma$ goes to 1, and that throughout the proof of the theorem in the present paper,
     we make no restriction upon the initial data $(\rho_{0},u_{0})$. Our result improves the one established by Li-Xu-Zhang in \cite{H.L. L},
      where, with small initial engergy, the existence of classical solution was proved.
\vspace{4mm}


\end{abstract}
\tableofcontents
\section{Introduction}

\qquad In this paper, we consider the following isentropic
 compressible magnetohydrodynamic equations in $\mathbb{R}^{3}$ (refer, e.g., \cite{Cabannes, Kulikovskiy}):
 \begin{equation}\label{2dbu-E1.1}
 \left\{
 \begin{array}{l}
   \D\rho_{t}+\mathrm{div}(\rho u)=0,\\[2mm]
 \D(\rho u)_{t}+\mathrm{div}(\rho u\otimes u)+\nabla P(\rho)=(\nabla\times H)\times H+\mu\Delta u+(\mu+\lambda)\nabla\mathrm{div}u,\\[2mm]
 \D H_{t}-\nabla \times(u\times H)=-\nabla\times(\nu\nabla\times H),\\[2mm]
 \D\mathrm{div}H=0,\ \ x\in \mathbb{R}^{3},\ \ t>0,
 \end{array}
 \right.
 \end{equation}
with the initial data
\begin{align}
 \D (\rho, u, H)(x,0)=(\rho_{0}, u_{0}, H_{0})(x),\ \ x\in \mathbb{R}^{3},
\end{align}
and the far-field behavior
\begin{align}
  \D (\rho, u, H)\rightarrow (0,0,0)\ \ \mathrm{as}\ \ |x|\rightarrow\infty,\ \ \mathrm{for}\ \  t>0.
\end{align}
Here $\rho=\rho(x,t)$, $u=(u^{1}, u^{2}, u^{3})(x,t)$, $P$ and $H=(H^{1}, H^{2}, H^{3})(x,t)$ represent the density, velocity, pressure and magnetic field of the fluid respectively. More precisely, $P$ is given by
\begin{align}
 \D P(\rho)=A\rho^{\gamma},
\end{align}
where $\gamma$ is the adiabatic exponent, and $A>0$ is a constant. Without loss of generality, we assumed that $A=1$. The viscosity coefficients $\mu$ and $\lambda$ satisfy
\begin{align}\label{2dbu-E1.5}
  \D \mu> 0,\ \ 3\lambda+2\mu\geq 0.
\end{align}
The constant $\nu>0$ is the resistivity coefficient which is inversely proportional
 to the electrical conductivity and acts as the magnetic diffusivity of magnetic fields.

 The magnetohydrodynamic $(\mathrm{MHD})$ model is used to study the dynamics
 of conducting fluid under the effect of the magnetic field and finds
its way in a huge range of physical objects, from liquid metals
to cosmic plasmas, refer for example \cite{Cabannes, Jeffrey, Landau, Kulikovskiy, Polovin}. And for so, there have
been a lot of literatures on the $\mathrm{MHD}$
system (\ref{2dbu-E1.1})-(\ref{2dbu-E1.5}), see for instance, \cite{G.-Q1, G.-Q2, ducomet,
 J. Fan1, J. Fan2, J. Fan3, X. Hu1, X. Hu2, X. Hu3, X. Hu4, kawashima1,
  kawashima2, kawashima3, kawashima4,
   Umeda, volpert, D. Wang, J.W. Zhang1, J.W. Zhang2} and
references therein. It should be noted that if $H=0$, i.e., there
is no electromagnetic effect, then (\ref{2dbu-E1.1}) becomes the
compressible Navier-Stokes equations, which has been widely studied,
refer for example \cite{cho1, cho2, s.j. D1, s.j. D2, hoff, Kazhi-Sheluk, lions, H.Y. Wen1, H.Y. Wen2, H.Y. Wen3} and the references therein. The main
difficulty in investigating the issues of well-posedness and
dynamical behaviors of $\mathrm{MHD}$ system is caused by the
strong coupling and interplay interaction between the fluid motion and the magnetic field.
Now let's recall briefly some results on the multi-dimensional
compressible $\mathrm{MHD}$ system, especially the ones that are
closely relative to our topic in the present paper. With large initial data,
the local strong solutions
to the compressible $\mathrm{MHD}$ equations were proved
in \cite{volpert} and \cite{J. Fan3} for the case $\rho_{0}>0$ and the case $\rho_{0}\geq 0$, respectively.
When the initial data are small perturbations of a given constant state in $H^{3}$ -norm,
 Kawashima in \cite{kawashima1} firstly established a result on the global existence of smooth
 solutions to the general electro-magneto-fluid equations in $\mathbb{R}^{2}$.
The global existence and time decay rate of smooth solutions to the linearized two-dimensional
compressible $\mathrm{MHD}$ equations was studied by Umeda, Kawashima and Shizuta in \cite{Umeda}.
Zhang and Zhao \cite{J.W. Zhang2} proved the optimal decay estimates of classical solutions to the compressible
$\mathrm{MHD}$ equations when the initial data are close to a nonvacuum equilibrium.
For the case that the initial density is allowed to vanish and even has compact support,
Li-Xu-Zhang \cite{H.L. L} established  a result on the existence and large-time behavior of classical solution
 with regular initial data, which are of small energy but possibly large oscillations,
  and constant state as far field density which may contain vacuum.

Before stating our main results, we firstly explain the notations and conventions used through this paper.\\

\noindent\textbf{Notations.}\\

\vspace{2mm}

 (i)  $\D \intt f=\intt f dx,\ \ \int_{0}^{T} f =\int_{0}^{T}f dt.$

 \vspace{2mm}

 (ii) For $1\leq r\leq \infty$, denote the $L^{r}$ spaces and the standard Sobolev spaces as follows:
 \begin{equation}
 \left\{
 \begin{array}{l}
  \D L^{r}=L^{r}(\mathbb{R}^{3}),\ \ D^{k,r}=\{u\in L_{loc}^{1}(\mathbb{R}^{3})|\nabla^{k}u\in L^{r}(\mathbb{R}^{3})\},\ \ \|u\|_{ D^{k,r}}=\|u\|_{L^{r}},\\[2mm]
  \D D^{1}=D^{1,2},\ \ W^{k,r}=W^{k,r}(\mathbb{R}^{3}),\ \ H^{k}=W^{k,2},\\[2mm]
  \D \dot{H}^{\beta}=\left\{u:\mathbb{R}^{3}\rightarrow \mathbb{R}\Big|\|u\|_{\dot{H}^{\beta}}^{2}=\int_{R^{3}}|\xi|^{2\beta}|\hat{u}(\xi)|^{2}d\xi<\infty\right\}.
 \end{array}
  \right.
 \end{equation}

 \vspace{2mm}

 (iii) $\D G\triangleq(2\mu+\lambda)\mdiv u-P-\frac{1}{2}|H|^{2}$ is the so-called effective viscous flux, while $\omega\triangleq \nabla\times u$ is the vorticity.

  \vspace{2mm}

 (iv) $\D \dot{h}=h_{t}+u\cdot\nabla h$ denotes the material derivatives.

  \vspace{2mm}

  (v) $\D E_{0}=\intt\left(\frac{1}{2}\rho_{0}|u_{0}|^{2}+\frac{1}{\gamma-1}\rho_{0}^{\gamma}+\frac{1}{2}|H_{0}|^{2}\right)$ is the initial energy.

  \bigbreak

  Now it is the place to state our main theorem.
 \begin{thm}\label{2dbu-T1.1}
 Assume that the initial data $(\rho_{0},u_{0},H_{0})$ satisfy
 \begin{equation}\label{2dbu-E1.9}
 \left\{
   \begin{array}{l}
     \D \frac{1}{2}\rho_{0}|u_{0}|^{2}+\frac{1}{\gamma-1}\rho_{0}^{\gamma}+\frac{1}{2}|H_{0}|^{2}\in L^{1},\ \ 0\leq \rho_{0}\leq\bar{\rho},\\[2mm]
     \D (\rho_{0},P(\rho_{0}))\in H^{2}\cap W^{2,q},\ \ u_{0}\in D^{1}\cap D^{2},\\[2mm]
     \D H_{0}\in \cap D^{1}\cap D^{2},\ \  \|\nabla u_{0}\|_{L^{2}}^{2}\leq M_{1},\ \  \D\|H_{0}\|_{D^{1}}^{2}\leq M_{2},\\[2mm]
     \D \|H_{0}\|_{L^{2}}^{2}\leq (\gamma-1)^{\frac{1}{6}}E_{0}^{\frac{1}{2}},
  \end{array}
  \right.
 \end{equation}
 for given constants $ M_{i}>0\ (i=1,2)$, $\bar{\rho}\geq 1$ and $q\in (3, 6)$,
  and that the compatibility condition holds
  \begin{align}\label{2dbu-E1.8}
   \D -\mu\Delta u_{0}-(\lambda+\mu)\nabla\mdiv u_{0}+\nabla P(\rho_{0})+\frac{1}{2}\nabla|H_{0}|^{2}-H_{0}\cdot\nabla H_{0}=\rho^{\frac{1}{2}}g,
  \end{align}
  with $g\in L^{2}$. In addition, we suppose that
  \begin{align}\label{2dbu-E1.011}
   \D (\gamma-1)^{\frac{1}{24}}E_{0}\leq 1,\ \ 1< \gamma\leq\frac{3}{2}.
  \end{align}
   Then, there exists a unique global classical solution $(\rho, u, H)$ in $\mathbb{R}^{3}\times [0,\infty)$ satisfying
  \begin{align}
    \D 0\leq \rho(x,t)\leq 2\bar{\rho},\  x\in \mathbb{R}^{3},\ t\geq 0,
  \end{align}
  and
  \begin{equation}
  \left\{
    \begin{array}{l}
       \D (\rho, P)\in C([0,T]; H^{2}\cap W^{2,q}),\\[2mm]
       \D u\in C([0,T]; D^{1}\cap D^{2})\cap L^{\infty}(\tau, T;D^{3}\cap D^{3,q}),\\[2mm]
       \D u_{t}\in L^{\infty}(\tau, T;D^{1}\cap D^{2})\cap H^{1}(\tau, T;D^{1}),\\[2mm]
       \D H\in C([0,T]; H^{2})\cap L^{\infty}(\tau, T;H^{3}),\\[2mm]
       \D H_{t}\in C([0,T]; L^{2})\cap H^{1}(\tau, T;L^{2}),
     \end{array}
     \right.
  \end{equation}
  for any $0< \tau< T<\infty$, provided that
\begin{align}
    \D (\gamma-1)^{\frac{1}{6}}E_{0}^{\frac{1}{2}}\leq \varepsilon\triangleq\min\left\{\left(\frac{\bar{\rho}}{2K_{9}}\right)^{16},\varepsilon_{5},\frac{\bar{\rho}}{4}\right\}.\non\\
\end{align}
  Here
  \begin{align}
  \D&\varepsilon_{5}=\min\left\{\varepsilon_{4},(C_{4}K_{9})^{-1},1\right\},\non\\
  \D&\varepsilon_{4}=\min\left\{\varepsilon_{3},\frac{1}{K_{6}^{54}},
  \left(C(\bar{\rho})(\gamma-1)^{\frac{1}{3}}K_{6}^{\frac{3}{2}}\right)^{-\frac{9}{8}},1\right\},\non\\
  \D&\varepsilon_{3}=\min\left\{\varepsilon_{1},\varepsilon_{2},4C_{3}(\bar{\rho}))^{-27},1\right\},\non\\
    \D &\varepsilon_{2}=\min\left\{\varepsilon_{1},(4C_{1})^{-\frac{27}{2}}+(4C_{2})^{-27}\right\},\non\\
    \D &\varepsilon_{1}=\min\left\{1,\left(4C(M_{2}^{\frac{3}{4}}+K_{1})\right)^{-9}\right\}.
  \end{align}
 \end{thm}
 \begin{rem}
 We make no restriction on the initial data $(\rho_{0},u_{0})$. In fact, it follows from (\ref{2dbu-E1.9}) that $\D E_{0}\leq C_{0}(\gamma-1)^{-\frac{1}{24}}$, then the upper bound of $E_{0}$ may go to $\infty$ as $\gamma$ goes to 1, in spite that $\|H_{0}\|_{L^{2}}$ is small.
 \end{rem}
  \begin{rem}The solution obtained in Theorem \ref{2dbu-T1.1} becomes a classical one away from the initial time. More precisely, we establish a result on the existence of a classical solution to (\ref{2dbu-E1.1})-(
 \ref{2dbu-E1.5}) under the assumption that $(\gamma-1)^{\frac{1}{6}}E_{0}^{\frac{1}{2}}$
 and  $\|H_{0}\|_{L^{2}}$ are
 suitably small. Moreover, we care more about the case that $\gamma$ is near 1,
  so the assumption $1< \gamma \leq\frac{3}{2}$ is reasonable. Indeed,
   the initial energy except the $L^{2}$-norm of $H_{0}$
    is allowed to be large when $\gamma$ is near 1. When the far-field density is vacuum,
     our result in Theorem \ref{2dbu-T1.1} is a generalization of that in \cite{H.L. L}. It should be emphasized that for the case $\gamma$ is some fixed constant, Theorem \ref{2dbu-T1.1} is still applicable ( necessarily after some modification for the proof ).
 \end{rem}
 \begin{rem}
 If we remove $\|\nabla u\|_{L^{2}}\leq M_{1}$ and $\|\nabla H\|_{L^{2}}\leq M_{2}$
 in (\ref{2dbu-E1.9}) in Theorem \ref{2dbu-T1.1}, and assume instead that
 $u_{0}, H_{0}\in \dot{H}^{\beta}(\beta\in (\frac{1}{2},1])$ with $\|u\|_{ \dot{H}^{\beta}}\leq \bar{M}_{1}$ and
 $\D \|H\|_{ \dot{H}^{\beta}}\leq \bar{M}_{2}$ for some $\bar{M}_{i}> 0\ (i=1,2)$, Theorem \ref{2dbu-T1.1} will still hold, and the $\varepsilon$ in Theorem \ref{2dbu-T1.1} will also depend on $\bar{M}_{i}$
  instead of $M_{i}$ correspondingly. This can be achieved by a similar way as in \cite{H.L. L}.
  \end{rem}
 \begin{rem}\label{2dbu-R0.1}
 It should be noted that when the viscous coefficient $\mu$ is taken to be suitable large, the initial energy except the $L^{2}$- norm of $H_{0}$
  could also be large, which together with the conclusion in Theorem \ref{2dbu-T1.1}, implies the fact that when $\D (\gamma-1)^{\frac{1}{6}}E_{0}^{\frac{1}{2}}\mu^{-\alpha_{1}}$
  and $\|H_{0}\|_{L^{2}}$ are
  suitably small for some $\alpha_{1}>0$,
 the existence of classical solutions to (\ref{2dbu-E1.1})-(\ref{2dbu-E1.5}) could also be obtained. And this can be done by using a
 similar method as in \cite{X.F. Hou}, which considered the compressible Navier-Stokes equations,
 we omit it for simplicity in the present paper. Moreover,  when $H=0$, i.e., there is no electromagnetic effect, (\ref{2dbu-E1.1})
  reduces to the compressible Navier-Stokes equations. Roughly speaking, we generalize the result of \cite{X.F. Hou} to the compressible $\mathrm{MHD}$
  equations.
 \end{rem}
 We now briefly make some comments on the analysis of the present paper. Note that
  the local existence and uniqueness of classical solutions to problem
 (\ref{2dbu-E1.1})-(\ref{2dbu-E1.5}) can be proved by combining the arguments in \cite{J. Fan3}
 with the higher order estimates in section 4 of \cite{H.L. L}. Hence,
 to extend the classical solution globally in time,
 we just need some global a priori estimates on the smooth solution $(\rho, u, H)$ in suitable regularity norms. Formally, the key to the proof
 is to get the time-independent upper bound of the density as well as the time-dependent higher norm estimates of $(\rho, u, H)$.
In this paper, the latter one follows in the same way as in \cite{H.L. L} (see Lemmas 4.1-4.6), once the former one is achieved. To derived the upper bound of the density, on the one hand, we try to adapt some basic ideas in \cite{hoff, X.D. Huang, H.L. L}. However,
new difficulties arise in our analysis, since the smallness of $(\gamma-1)^{\frac{1}{6}}E_{0}^{\frac{1}{2}}$ does not result in the small initial energy. One the other hand, compared with compressible Navier-Stokes equations, the strong coupling and interplay interaction between the fluid motion and the magnetic field, such as $\nabla\times(u\times H)$ and $(\nabla \times H)\times H$, will bring out some new difficulties.

Precisely, in \cite{hoff, X.D. Huang, H.L. L} the smallness of the initial energy was used to ensure the smallness of $\int_{0}^{T}\intt|\nabla u|^{2}$ and $\|H_{0}\|_{L^{2}}$, which play crucial role in the proof of the upper bound of density. Similar to \cite{X.F. Hou, X.D. Huang, H.L. L}, here we need to close
 the a priori estimates $A_{1}$ and $A_{2}$. Compared with \cite{X.F. Hou, X.D. Huang}, we not only need to handle
 the terms $|\nabla u|^{2}$, $|\nabla u|^{3}$, $|\nabla u|^{4}$, $P|\nabla u|^{2}$, $|P\nabla u|^{2}$,
 but the terms caused by $\nabla\times(u\times H)$ and $(\nabla \times H)\times H$,
 like $H\cdot\nabla H\cdot u$ and $\nabla|H|^{2}\cdot u$.
 Adapting the idea developed in \cite{X.F. Hou}, we need to derive the smallness of
$\int_{0}^{\sigma(T)}|\nabla u|^{2}$, but this is not trivial because of the lack
 of the smallness of $\|H_{0}\|_{L^{2}}$. The key observation to overcome this
 difficulty is as follows: Looking back to the basic energy $E_{0}=\intt\left(\frac{1}{2}\rho_{0}|u_{0}|^{2}+\frac{1}{\gamma-1}\rho_{0}^{\gamma}+\frac{1}{2}|H_{0}|^{2}\right)$, the smallness of
 of $\gamma-1$ could remove the smallness restriction upon $\rho_{0}$ and the term involving $u_{0}$ and $\rho_{0}$. But it
 has nothing to do with $H_{0}$. Moreover, For all terms in (\ref{2dbu-E1.1}), we can never see any term in which $\rho$ is coupled with $H$.
 Hence we assume that $\|H_{0}\|_{L^{2}}\leq (\gamma-1)^{\frac{1}{6}}E_{0}$ is small, and then we succeed to derive some estimates on the smallness and boundedness of $H$ and its derivatives with a key estimate $\int_{0}^{T}\|\nabla u\|_{L^{2}}^{4}$. Similar to \cite{X.F. Hou, X.D. Huang}, we try to estimate $A_{1}$ and $A_{2}$
 and achieve an inequality involving $|\nabla u|^{2}$, $|\nabla u|^{3}$, $|\nabla u|^{4}$, $P|\nabla u|^{2}$ and $|P\nabla u|^{2}$. And then we handle all these terms one the right hand side of the inequality with two crucial boundedness estimate (see Lemma \ref{2dbu-L3.7}). Thus the upper bound of $\rho$ is obtained by a standard method as in \cite{X.F. Hou, X.D. Huang,H.L. L}, together with some new estimate (see (\ref{2dbu-E3.01117}) and (\ref{2dbu-E3.122})). It should be noted that during the process, the estimates obtained for $H$ always play a key role, especially when controlling the coupled term  $\nabla\times(u\times H)$.

 The rest of the paper is organized as follows. In section 2, we first collect some elementary inequalities and facts which will be need in the later analysis. In section 3, we devote to derive the necessary lower-order a priori estimates on the classical solution which is independent of time. The time-dependent estimates on the higher-norms of the solutions will be proved in Section 4, and then Theorem \ref{2dbu-T1.1} is proved.

 \section{Preliminaries}
 \qquad In this section, we will recall some elementary inequality and results which will be used used frequently later. We begin with the following well-known Gagliardo-Nirenberg inequality (see \cite{Ladyzenskaja}).
 \begin{lem}\label{2dbu-L2.1}
   For $2\leq p\leq 6$, $1<q<\infty$, and $3<r<\infty$, there exists a generic constant $C> 0$, depending only on $q$ and $r$, such that for
   $f\in H^{1}$ and $g \in L^{q}\cap D^{1,r}$, we have
   \begin{align}
     \D&\|f\|_{L^{p}}\leq C\|f\|_{L^{2}}^{\frac{6-p}{2p}}\|\nabla f\|_{L^{2}}^{\frac{3p-6}{2p}},\label{2dbu-E2.01}\\
     \D&\|g\|_{L^{\infty}}\leq C\|g\|_{L^{q}}^{\frac{q(r-3)}{3r+q(r-3)}}\|g\|_{L^{r}}^{\frac{3r}{3r+q(r-3)}}.\label{2dbu-E2.02}
   \end{align}
 \end{lem}
 Similar to the compressible Navier-Stokes equations (see, for example\cite{X.F. Hou, X.D. Huang}), one can easily derive the following elliptic equations from (\ref{2dbu-E1.1}):
 \begin{align}\label{2dbu-E2.03}
 \D \Delta G=\mdiv (\rho \dot{u})-\mdiv\mdiv(H\otimes H),\ \ \mu\Delta \omega=\nabla\times (\rho\dot{u}-\mdiv(H\otimes H)).
 \end{align}
 We now state some elementary $L^{p}$-estimates for the elliptic equations in (\ref{2dbu-E2.03}) by the virtue of (\ref{2dbu-E2.01}).
 \begin{lem}\label{2dbu-L2.2}
 Let $(\rho, u, H)$ be a smooth solution to (\ref{2dbu-E1.1})-(\ref{2dbu-E1.5}) on $\mathbb{R}^{3}\times (0,T]$. Then there
 exists a generic $C>0$, which may depend on $\mu$ and $\lambda$, such that for any $p\in [2,6]$,
 \begin{align}
   \D \|\nabla G\|_{L^{p}}&+\|\nabla \omega\|_{L^{p}}\leq C\left(\|\rho\dot{u}\|_{L^{p}}+\|H\cdot\nabla H\|_{L^{p}}\right),\label{2dbu-E2.04}\\[2mm]
   \D\| G\|_{L^{6}}&+\| \omega\|_{L^{6}}\leq C\left(\|\rho\dot{u}\|_{L^{2}}+\|H\cdot\nabla H\|_{L^{2}}\right),\label{2dbu-E2.05}\\[2mm]
   \D\|\nabla u\|_{L^{6}}&\leq C(\|\rho\dot{u}\|_{L^{2}}+\|P\|_{L^{6}}+\|H\cdot\nabla H\|_{L^{2}}),\label{2dbu-E2.06}\\[2mm]
   \D\|\nabla u\|_{L^{4}}&\leq C\|\nabla u\|_{L^{2}}^{\frac{1}{4}}\|\rho\dot{u}\|_{L^{2}}^{\frac{3}{4}}
  +C\|P\|_{L^{2}}^{\frac{1}{4}}\|\rho\dot{u}\|_{L^{2}}^{\frac{3}{4}}
  +C\|H\|_{L^{4}}^{\frac{1}{2}}\|\rho\dot{u}\|_{L^{2}}^{\frac{3}{4}}\non\\[2mm]
  \D&\quad+C\|\nabla u\|_{L^{2}}^{\frac{1}{4}}\|H\cdot\nabla H\|_{L^{2}}^{\frac{3}{4}}
  +C\|P\|_{L^{2}}^{\frac{1}{4}}\|H\cdot\nabla H\|_{L^{2}}^{\frac{3}{4}}\non\\[2mm]
  \D&\quad+C\|H\|_{L^{4}}^{\frac{1}{2}}\|H\cdot\nabla H\|_{L^{2}}^{\frac{3}{4}}+C\|P\|_{L^{4}}.\label{2dbu-E2.07}
 \end{align}
 \end{lem}
 \begin{proof}
 The proof of inequalities (\ref{2dbu-E2.04})-(\ref{2dbu-E2.06}) can be found in Lemma 2.2 in \cite{H.L. L}. Here we will prove (\ref{2dbu-E2.07}). It follows from direct computation that
 \begin{align}
  \D -\Delta u=-\nabla\mdiv u+\nabla\times \omega,
 \end{align}
 then the standard $L^{p}$-estimate for elliptic equation, together with (\ref{2dbu-E2.01}) and (\ref{2dbu-E2.04}), leads to
 \begin{align}\label{2dbu-E2.09}
   \D\quad\|\nabla u\|_{L^{p}}&\leq C\left(\|\mdiv u\|_{L^{p}}+\|w\|_{L^{p}}\right)\non\\[2mm]
   \D&\leq C\left(\|G\|_{L^{p}}+\|P\|_{L^{p}}+\||H|^{2}\|_{L^{p}}+\|\nabla u\|_{L^{2}}^{\frac{6-p}{2p}}\Big(\|\rho\dot{u}\|_{L^{2}}^{\frac{3p-6}{2p}}+\|H\cdot\nabla H\|_{L^{2}}^{\frac{3p-6}{2p}}\Big)\right)\non\\[2mm]
   \D&\leq C\|\nabla u\|_{L^{2}}^{\frac{6-p}{2p}}\|\rho\dot{u}\|_{L^{2}}^{\frac{3p-6}{2p}}
  +C\|P\|_{L^{2}}^{\frac{6-p}{2p}}\|\rho\dot{u}\|_{L^{2}}^{\frac{3p-6}{2p}}
  +C\|H\|_{L^{4}}^{\frac{2(6-p)}{2p}}\|\rho\dot{u}\|_{L^{2}}^{\frac{3p-6}{2p}}\non\\[2mm]
  \D&\quad+C\|\nabla u\|_{L^{2}}^{\frac{6-p}{2p}}\|H\cdot\nabla H\|_{L^{2}}^{\frac{3p-6}{2p}}
  +C\|P\|_{L^{2}}^{\frac{6-p}{2p}}\|H\cdot\nabla H\|_{L^{2}}^{\frac{3p-6}{2p}}\non\\[2mm]
  \D&\quad+C\|H\|_{L^{4}}^{\frac{2(6-p)}{2p}}\|H\cdot\nabla H\|_{L^{2}}^{\frac{3p-6}{2p}}+C\|P\|_{L^{p}}.
 \end{align}
 Let $p=4$ in (\ref{2dbu-E2.09}), one gets (\ref{2dbu-E2.07}).
 \end{proof}
  To obtain the uniform (in time) upper bound of the density, we need the following Zlotnik inequality.
\begin{lem}[see\cite{A.Z}]
Assume that the function $y$ satisfies
\begin{align}
 \D y'(t)=g(y)+b'(t)\ \ \mathrm{on}\ \  [0,T],\  y(0)=y^{0},
\end{align}
with $g\in C(\mathbb{R})$ and $y, b\in W^{1,1}(0,T)$. If $g(\infty)=-\infty$ and
\begin{align}
 \D b(t_{2})-b(t_{1})\leq N_{0}+N_{1}(t_{2}-t_{1}),
\end{align}
for all $0\leq t_{1}\leq t_{2}\leq T$ with some $N_{0}\geq 0$ and $N_{1}\geq 0$, then
\begin{align}
 \D g(\xi)\leq -N_{1},\ \ for\ \ \xi\geq \bar{\xi}.
\end{align}
\end{lem}

\section{Time-independent estimates}
\qquad In this section, we will derive the uniform time-independent estimates of the solution to (\ref{2dbu-E1.1})-(\ref{2dbu-E1.5})
and the time-independent upper bound of the density. Assume that $(\rho, u, H)$ is a smooth solution
 to (\ref{2dbu-E1.1})-(\ref{2dbu-E1.5}) on $\mathbb{R}^{3}\times (0,T)$ for some positive time $T> 0$. Set
$\sigma=\sigma(t)\triangleq \min\{1,t\}$ and define the following functionals:
\begin{align}\label{2dbu-E3.01}
  \D A_{1}(T)&\triangleq\sup_{0\leq t\leq T}\sigma\intt\Big(|\nabla u|^{2}+|\nabla H|^{2}\Big)\non\\[2mm]
  &\D\quad+\int_{0}^{T}\sigma\left(\|\rho^{\frac{1}{2}}\dot{u}\|_{L^{2}}^{2}+\|\nabla^{2}H\|_{L^{2}}^{2}+
  \|H_{t}\|_{L^{2}}^{2}\right),\nonumber\\[2mm]
  \D A_{2}(T)&\triangleq\sup_{0\leq t\leq T}\sigma^{2}\intt\Big(\rho|\dot{u}|^{2}+|\nabla^{2} H|^{2}+|H_{t}|^{2}\Big)\non\\[2mm]
  &\D\quad+\int_{0}^{T}\sigma^{2}\Big(\|\nabla\dot{ u}\|_{L^{2}}^{2}+\|\nabla H_{t}\|_{L^{2}}^{2}\Big),\nonumber\\[2mm]
  \D A_{3}(T)&\triangleq \sup_{0\leq t\leq T}\|H\|_{L^{3}}^{3}+\int_{0}^{T}\intt|H||\nabla H|^{2},\nonumber\\[2mm]
 \D A_{4}(T) &\triangleq\sup_{0\leq t\leq T}\sigma^{\frac{1}{3}}\|\nabla u\|_{L^{2}}^{2},
  \nonumber\\[2mm]
    \D A_{5}(T)&\triangleq\sup_{0\leq t\leq T}\intt\rho|u|^{3}.
\end{align}
Throughout this section, for simplicity we denote by $C$ or $C_{i}$ $(i=1,2,\cdots)$ the generic positive constants which may depend on $\mu$, $\lambda$,
$\nu$, $A$, $\gamma$, $\bar{\rho}$, $\tilde{\rho}$, $C_{0}$, $M_{i}\ (i=1,2)$ and $\|\rho_{0}\|_{L^{1}}$ but independent of time $T>0$ and $\gamma-1$. Sometimes $C(\alpha)$ is also used to emphasize the dependent of $\alpha$. we state the key proposition in the present paper as follows.
\begin{prop}\label{2dbu-P3.1}
  Assume that the initial data $(\rho_{0}, u_{0}, H_{0})$ satisfy (\ref{2dbu-E1.9})-(\ref{2dbu-E1.011}). Let $(\rho, u, H)$
  be a smooth solution to problem (\ref{2dbu-E1.1})-(\ref{2dbu-E1.5}) on $\mathbb{R}^{3}\times (0,T]$ satisfying
  \begin{equation}\label{2dbu-E3.02}
  \left\{
  \begin{array}{l}
   \D 0\leq \rho(x,t)\leq 2\bar{\rho},\ (x,t)\in \mathbb{R}^{3}\times [0,T],\\
   \D A_{1}(T)+A_{2}(T)\leq 2\left((\gamma-1)^{\frac{1}{6}}E_{0}^{\frac{1}{2}}\right)^{\frac{1}{2}},\ \
   A_{3}(T)\leq 2\left((\gamma-1)^{\frac{1}{6}}E_{0}^{\frac{1}{2}}\right)^{\frac{1}{9}},\\
   A_{4}(\sigma(T))+A_{5}(\sigma(T))\leq 2\left((\gamma-1)^{\frac{1}{6}}E_{0}^{\frac{1}{2}}\right)^{\frac{1}{9}},
  \end{array}
  \right.
  \end{equation}
  then
  \begin{equation}
  \left\{
  \begin{array}{l}
   \D 0\leq \rho(x,t)\leq \frac{7}{4}\bar{\rho},\ (x,t)\in \mathbb{R}^{3}\times [0,T],\\
   \D A_{1}(T)+A_{2}(T)\leq \left((\gamma-1)^{\frac{1}{6}}E_{0}^{\frac{1}{2}}\right)^{\frac{1}{2}},\ \
   A_{3}(T)\leq \left((\gamma-1)^{\frac{1}{6}}E_{0}^{\frac{1}{2}}\right)^{\frac{1}{9}},\\
   A_{4}(\sigma(T))+A_{5}(\sigma(T))\leq \left((\gamma-1)^{\frac{1}{6}}E_{0}^{\frac{1}{2}}\right)^{\frac{1}{9}},
  \end{array}
  \right.
  \end{equation}
   provided that
   \begin{align}
    \D (\gamma-1)^{\frac{1}{6}}E_{0}^{\frac{1}{2}}\leq \varepsilon\triangleq\min\left\{\left(\frac{\bar{\rho}}{2K_{9}}\right)^{16},\varepsilon_{5},\frac{\bar{\rho}}{4}\right\}.\non\\
\end{align}
  Here
  \begin{align}
  \D&\varepsilon_{5}=\min\left\{\varepsilon_{4},(C_{4}K_{9})^{-1},1\right\},\non\\
  \D&\varepsilon_{4}=\min\left\{\varepsilon_{3},\frac{1}{K_{6}^{54}},
  \left(C(\bar{\rho})(\gamma-1)^{\frac{1}{3}}K_{6}^{\frac{3}{2}}\right)^{-\frac{9}{8}},1\right\},\non\\
  \D&\varepsilon_{3}=\min\left\{\varepsilon_{1},\varepsilon_{2},4C_{3}(\bar{\rho}))^{-27},1\right\},\non\\
    \D &\varepsilon_{2}=\min\left\{\varepsilon_{1},(4C_{1})^{-\frac{27}{2}}+(4C_{2})^{-27}\right\},\non\\
    \D &\varepsilon_{1}=\min\left\{1,\left(4C(M_{2}^{\frac{3}{4}}+K_{1})\right)^{-9}\right\}.
  \end{align}
\end{prop}
\begin{proof}
 Proposition \ref{2dbu-P3.1} can be derived from Lemmas \ref{2dbu-L3.1}-\ref{2dbu-L3.10} below.
\end{proof}
\begin{lem}\label{2dbu-L3.1}
Under the same assumption as in Proposition \ref{2dbu-P3.1}, we have
\begin{align}
 &\D \sup_{0\leq t\leq T}\intt P\leq(\gamma-1)E_{0},\label{2dbu-E3.03}\\
 &\D \int_{0}^{T}\intt|\nabla u|^{2}\leq \frac{E_{0}}{\mu}.\label{2dbu-E3.04}
\end{align}
\end{lem}
\begin{proof}
Multiplying $(\ref{2dbu-E1.1})_{1}$, $(\ref{2dbu-E1.1})_{2}$ and $(\ref{2dbu-E1.1})_{3}$ by $\frac{\gamma\rho^{\gamma}}{\gamma-1}$, $u$
 and $H$, respectively, and integrating the resulting equation over $\mathbb{R}^{3}\times (0,T]$, we have
 \begin{align}
  \D\sup_{0\leq t\leq T} \intt\Big(\frac{P}{\gamma-1}&+\frac{1}{2}\rho|u|^{2}+\frac{1}{2}|H|^{2}\Big)\non\\
  &\D+\int_{0}^{T}\intt \left(\mu|\nabla u|^{2}+(\mu+\lambda)|\mathrm{div} u|^{2}+\nu|\nabla H|^{2}\right)\leq E_{0},
\end{align}
which gives (\ref{2dbu-E3.03}) and (\ref{2dbu-E3.04}).
\end{proof}
\begin{lem}\label{2dbu-L3.02}
Under the same assumption as in Proposition \ref{2dbu-P3.1}, it holds that
\begin{align}
 \D \int_{0}^{T}\|\nabla u\|_{L^{2}}^{4}&\leq CK_{1}\left((\gamma-1)^{\frac{1}{6}}E_{0}^{\frac{1}{2}}\right)^{\frac{1}{9}},
\end{align}
 where $\D K_{1}=\left(1+(\gamma-1)^{\frac{1}{144}}\right)$.
\end{lem}
\begin{proof}
\begin{align}
  \D \int_{0}^{T}\|\nabla u\|_{L^{2}}^{4}&\leq \int_{0}^{\sigma(T)}\|\nabla u\|_{L^{2}}^{4}+\int_{\sigma(T)}^{T}\sigma\|\nabla u\|_{L^{2}}^{4}\nonumber\\
  &\D\leq \sup_{0\leq t\leq \sigma(T)}\left(\sigma^{\frac{1}{3}}\|\nabla u\|_{L^{2}}^{2}\right)^{2}\int_{0}^{\sigma(T)}\sigma^{-\frac{2}{3}}\nonumber\\
  &\quad+C\bi(\sup_{ \sigma(T)\leq t\leq T}\sigma\|\nabla u\|_{L^{2}}^{2}\bi)\int_{\sigma(T)}^{T}\|\nabla u\|_{L^{2}}^{2}\nonumber\\
  &\D\leq C\left((\gamma-1)^{\frac{1}{6}}E_{0}^{\frac{1}{2}}\right)^{\frac{2}{9}}
  +\left((\gamma-1)^{\frac{1}{6}}E_{0}^{\frac{1}{2}}\right)^{\frac{1}{2}}E_{0}\nonumber\\
  &\D\leq C\left((\gamma-1)^{\frac{1}{6}}E_{0}^{\frac{1}{2}}\right)^{\frac{2}{9}}+
  C\left((\gamma-1)^{\frac{1}{6}}E_{0}^{\frac{1}{2}}\right)^{\frac{1}{6}}
  \left((\gamma-1)^{\frac{1}{6}}E_{0}^{\frac{1}{2}}\right)^{\frac{1}{3}}E_{0}\non\\
  &\D \leq C\left((\gamma-1)^{\frac{1}{6}}E_{0}^{\frac{1}{2}}\right)^{\frac{2}{9}}+ C\left(1+(\gamma-1)^{\frac{1}{144}}\right)\left((\gamma-1)^{\frac{1}{6}}E_{0}^{\frac{1}{2}}\right)^{\frac{1}{9}}\non\\
  &\D \leq K_{1}\left((\gamma-1)^{\frac{1}{6}}E_{0}^{\frac{1}{2}}\right)^{\frac{1}{9}},
\end{align}
here $(\gamma-1)^{\frac{1}{6}}E_{0}^{\frac{1}{2}}\leq 1$ and (\ref{2dbu-E1.011}) have been used. Lemma \ref{2dbu-L3.02} is proved.
\end{proof}
\begin{lem}\label{2dbu-L3.3}
Under the same assumption as in Proposition \ref{2dbu-P3.1}, it holds that
\begin{align}\label{2dbu-E3.08}
  \D \sup_{0\leq t\leq T}\Big(\|H\|_{L^{2}}^{2}&+\sigma\|\nabla H\|_{L^{2}}^{2}\Big)\non\\
  \D&+\int_{0}^{T}\left(\|\nabla H\|_{L^{2}}^{2}+\sigma\|H_{t}\|_{L^{2}}^{2}+\sigma\|\nabla^{2}H\|_{L^{2}}^{2}\right)\leq CK_{2}^{2}(\gamma-1)^{\frac{1}{6}}E_{0}^{\frac{1}{2}}
\end{align}
and
\begin{align}\label{2dbu-E3.9}
  \sup_{0\leq t\leq T}\|\nabla H\|_{L^{2}}^{2}+\int_{0}^{T}\left(\|H_{t}\|_{L^{2}}^{2}+\|\nabla^{2}H\|_{L^{2}}^{2}\right)\leq CK_{2}M_{2}.
\end{align}
where $K_{2}=e^{2K_{1}}$.
\end{lem}
\begin{proof}
Multiplying $(\ref{2dbu-E1.1})_{3}$ by $H_{t}$, then integrating over $\mathbb{R}^{3}$, using $(\ref{2dbu-E1.1})_{4}$, $\mathrm{H\ddot{o}lder}$ inequality and Cauchy inequality, we get
\begin{align}
  \D \intt H\cdot H_{t}&=\intt \nabla\times(u\times H)\cdot H-\intt \nabla\times(\nu\nabla\times H)\cdot H\non\\[2mm]
  \D&=\intt\left[(H\cdot\nabla)u-(u\cdot\nabla)H-(\mdiv u)H\right]\cdot H+\intt\Delta H\cdot H\non\\[2mm]
  \D&\leq \intt |u||\nabla H||H|-\nu\intt |\nabla u|^{2}\non\\[2mm]
  \D&\D \leq C \|u\|_{L^{6}}\|\nabla H\|_{L^{2}}\|H\|_{L^{3}}-\nu\intt |\nabla u|^{2}\non\\[2mm]
  \D& \leq\epsilon \|\nabla H\|_{L^{2}}^{2}+C\|\nabla u\|_{L^{2}}^{4}\|H\|_{L^{2}}^{2}-\nu\intt |\nabla u|^{2},
\end{align}
which implies that
\begin{align}
  \rf \intt |H|^{2}+\intt \nu|\nabla H|^{2}\leq C\|\nabla u\|_{L^{2}}^{4}\|H\|_{L^{2}}^{2}.
\end{align}
An application of Gronwall's inequality leads to
\begin{align}\label{2dbu-E3.013}
  \D \|H\|_{L^{2}}^{2}+\int_{0}^{T}\nu\|\nabla H\|_{L^{2}}^{2}\leq CK_{2}\|H_{0}\|_{L^{2}}^{2}\leq Ce^{K_{1}}(\gamma-1)^{\frac{1}{6}}E_{0}^{\frac{1}{2}}.
\end{align}
Thanks to $(\ref{2dbu-E1.1})_{3}$ and Lemma \ref{2dbu-L2.1}, using integration by parts, we derive that
\begin{align}
  &\D\quad\rf\intt|\nabla H|^{2}+\intt|H_{t}|^{2}+\intt|\nabla^{2}H|^{2}\non\\
  &\D=\intt|H_{t}-\triangle H|^{2}\non\\
  &\D=\intt|H\cdot\nabla u-u\cdot\nabla H-H\mdiv u|^{2}\non\\
  &\D\leq C\|\nabla u\|_{L^{2}}^{4}\|\nabla H\|_{L^{2}}^{2}+\frac{1}{2}\|\nabla^{2}H\|_{L^{2}}^{2},
\end{align}
 we consequently have
\begin{align}\label{2dbu-E1.01}
  \rf\|\nabla H\|_{L^{2}}^{2}+\left(\|H_{t}\|_{L^{2}}^{2}+
  \|\nabla^{2}H\|_{L^{2}}^{2}\right)\leq C\|\nabla u\|_{L^{2}}^{4}\|\nabla H\|_{L^{2}}^{2}.
\end{align}
As before, Gronwall's inequality leads to
\begin{align}\label{2dbu-E3.016}
  \sup_{0\leq t\leq T}\|\nabla H\|_{L^{2}}^{2}+\int_{0}^{T}\left(\|H_{t}\|_{L^{2}}^{2}+\|\nabla^{2}H\|_{L^{2}}^{2}\right)\leq Ce^{K_{1}} M_{2}.
\end{align}
Multiplying (\ref{2dbu-E1.01}) by $ \sigma$, one has
\begin{align}
 \rf\left(\sigma\|\nabla H\|_{L^{2}}^{2}\right)+\sigma\left(\|H_{t}\|_{L^{2}}^{2}
 +\|\nabla^{2}H\|_{L^{2}}^{2}\right)\leq C\sigma\|\nabla u\|_{L^{2}}^{4}\|\nabla H\|_{L^{2}}^{2}
 +\sigma'\|\nabla H\|_{L^{2}}^{2}.
\end{align}
Again, using Gronwall's inequality, we get
\begin{align}\label{2dbu-E3.018}
  \D&\quad\sup_{0\leq t\leq T}\left(\sigma\|\nabla H\|_{L^{2}}^{2}\right)
  +\int_{0}^{T}\sigma\left(\|H_{t}\|_{L^{2}}^{2}+\|\nabla^{2}H\|_{L^{2}}^{2}\right)\non\\
  \D&\leq Ce^{K_{1}}\int_{0}^{T}\|\nabla H\|_{L^{2}}^{2}\non\\
   \D&\leq Ce^{2K_{1}}\|H_{0}\|_{L^{2}}^{2}.
\end{align}
Combining (\ref{2dbu-E3.013}), (\ref{2dbu-E3.016}) and (\ref{2dbu-E3.018}), we finish the proof of Lemma \ref{2dbu-L3.3}.
\end{proof}
\begin{lem}\label{2dbu-L3.4}
Under the same assumption as in Proposition \ref{2dbu-P3.1}, it holds that
\begin{align}\label{2dbu-E3.019}
 \D\sup_{0\leq t\leq T}\|H\|_{L^{3}}^{3}+\int_{0}^{T}\left(\||H|^{\frac{1}{2}}|\nabla H|\|_{L^{2}}^{2}+\|H\|_{L^{9}}^{3}\right)
 \leq \left((\gamma-1)^{\frac{1}{6}}E_{0}^{\frac{1}{2}}\right)^{\frac{1}{9}},
\end{align}
provided $\D (\gamma-1)^{\frac{1}{6}}E_{0}^{\frac{1}{2}}\leq \varepsilon_{1} $, where
\begin{align}
  \D \varepsilon_{1}=\min\left\{\left(4C(M_{2}^{\frac{3}{4}}+K_{1})\right)^{-9},1\right\}.
\end{align}
\end{lem}
\begin{proof}
 Multiplying $(\ref{2dbu-E1.1})_{3}$ by $3|H|H$, using integration by parts as in \cite{H.L. L}, we have
\begin{align}\label{2dbu-E3.021}
 \D&\quad\rf\intt |H|^{3}+3\nu\intt|H||\nabla H|^{2}+3\nu\intt |H||\nabla|H||^{2}\non\\
 \D&\leq \nu\intt|H||\nabla H|^{2}+\nu\intt |H||\nabla|H||^{2} +C\|\nabla u\|_{L^{2}}^{4}\|H\|_{L^{3}}^{3}.
\end{align}
Noticing that
\begin{align}
 \D &\|H\|_{L^{9}}^{3}\leq C\||H|^{\frac{3}{2}}\|_{L^{6}}^{2}\leq C\||\nabla H||H|^{\frac{1}{2}}\|_{L^{2}}^{2},\label{2dbu-E3.022}\\
 \D&\|H\|_{L^{\frac{9}{2}}}\leq C\|H\|_{L^{3}}^{\frac{1}{2}}\|H\|_{L^{9}}^{\frac{1}{2}}\leq C\|H\|_{L^{3}}^{\frac{1}{2}}\|\nabla |H|^{\frac{3}{2}}\|_{L^{2}}^{\frac{1}{3}},\label{2dbu-E3.023}
\end{align}
substituting (\ref{2dbu-E3.022}) and (\ref{2dbu-E3.023}) into (\ref{2dbu-E3.021}), using Cauchy inequality, we thus deduce that
\begin{align}\label{2dbu-E3.024}
  \D\rf\intt |H|^{3}+\intt|H||\nabla H|^{2}\leq C\left((\gamma-1)^{\frac{1}{6}}E_{0}^{\frac{1}{2}}\right)^{\frac{1}{9}}\|\nabla u\|_{L^{2}}^{4}.
\end{align}
Integrating (\ref{2dbu-E3.024}) over $[0,T]$, by the virtue of Sobolev embedding inequality, one can derived that
\begin{align}\label{2dbu-E3.025}
  &\quad\D\sup_{0\leq t\leq T}\|H\|_{L^{3}}^{3}+\int_{0}^{T}\intt|H||\nabla H|^{2}\non\\[2mm]
  &\D \leq C\|H_{0}\|_{L^{2}}^{\frac{3}{2}}\|H_{0}\|_{D^{1}}^{\frac{3}{2}}
  +CK_{1}\left((\gamma-1)^{\frac{1}{6}}E_{0}^{\frac{1}{2}}\right)^{\frac{2}{9}}\non\\[2mm]
  &\D\leq CM_{2}^{\frac{3}{4}}\left((\gamma-1)^{\frac{1}{6}}E_{0}^{\frac{1}{2}}\right)^{\frac{3}{4}}
  +CK_{1}\left((\gamma-1)^{\frac{1}{6}}E_{0}^{\frac{1}{2}}\right)^{\frac{2}{9}}\non\\[2mm]
  &\D \leq C(M_{2}^{\frac{3}{4}}+K_{1})\left((\gamma-1)^{\frac{1}{6}}E_{0}^{\frac{1}{2}}\right)^{^{\frac{2}{9}}}\non\\
  &\D\leq \left((\gamma-1)^{\frac{1}{6}}E_{0}^{\frac{1}{2}}\right)^{^{\frac{1}{9}}},
\end{align}
provided
\begin{align}
  \D (\gamma-1)^{\frac{1}{6}}E_{0}^{\frac{1}{2}}<\min\left\{\left(4C(M_{2}^{\frac{3}{4}}+K_{1})\right)^{-9},1\right\}\triangleq \varepsilon_{1}.
\end{align}
Then estimate (\ref{2dbu-E3.025}), together with (\ref{2dbu-E3.022}), yields (\ref{2dbu-E3.019}). This ends up the proof of Lemma \ref{2dbu-L3.4}.
\end{proof}
\begin{lem}\label{2dbu-L3.5}
Under the same assumption as in Lemma \ref{2dbu-L3.4}, we have
\begin{align}\label{2dbu-E3.028}
  &\D\quad\sup_{0\leq t\leq \sigma(T)}\|\rho^{\frac{1}{2}}u\|_{L^{2}}^{2}
  +\int_{0}^{\sigma(T)}\intt|\nabla u|^{2}
  \leq  CK_{3}(\gamma-1)^{\frac{1}{6}}E_{0}^{\frac{1}{2}},
\end{align}
where $\D K_{3}=C(\bar{\rho})\left(1+(\gamma-1)^{\frac{4}{9}}+(\gamma-1)^{\frac{2}{3}}\right)$.
\end{lem}
\begin{proof}
Multiplying $(1.2)$ by $u$ and then integrating the resulting equality over $\mathbb{R}^{3}$, and using integration by parts, we have
\begin{align}\label{2dbu-E3.027}
  \D&\quad \frac{1}{2}\rf\intt \rho|u|^{2}+\intt\left(\mu|\nabla u|^{2}+(\lambda+\mu)|\mdiv u|^{2}\right)\non\\
  \D&=\intt P\mdiv u+\intt\left(H\cdot\nabla H-\frac{1}{2}\nabla|H|^{2}\right)u.
\end{align}
Integrating (\ref{2dbu-E3.027}) over $(0,\sigma(T))$, one has
\begin{align}
  &\D\quad\frac{1}{2}\sup_{0\leq t\leq \sigma(T)}\|\rho^{\frac{1}{2}}u\|_{L^{2}}^{2}
  +\int_{0}^{\sigma(T)}\intt\left(\frac{\mu}{2}|\nabla u|^{2}
  +(\mu+\lambda)|\mdiv u|^{2}\right)\non\\[2mm]
  &\D\leq \frac{1}{2}\intt\rho_{0}|u_{0}|^{2}+\int_{0}^{\sigma(T)}\intt P\mdiv u
  +\int_{0}^{\sigma(T)}\intt\left[(H\cdot\nabla H-\frac{1}{2}\nabla|H|^{2})u\right]\non\\[2mm]
  &\D\leq C\|\rho_{0}\|_{L^{\frac{3}{2}}}\|u_{0}\|_{L^{6}}
  +\frac{\mu}{6}\int_{0}^{\sigma(T)}\intt|\nabla u|^{2}+C\int_{0}^{\sigma(T)}\intt P^{2}
  +\int_{0}^{\sigma(T)}\|u\|_{L^{6}}\|H\|_{L^{3}}\|\nabla H\|_{L^{2}}^{2}\non\\[2mm]
  &\D\leq C\left((\gamma-1)E_{0}\right)^{\frac{2}{3}}
  +\frac{\mu}{4}\int_{0}^{\sigma(T)}\intt|\nabla u|^{2}+C(\bar{\rho})(\gamma-1)E_{0}+
  \int_{0}^{\sigma(T)}\|H\|_{L^{3}}^{2}\|\nabla H\|_{L^{2}}^{4}dt\non\\[2mm]
  &\D\leq C\left((\gamma-1)E_{0}\right)^{\frac{2}{3}}
  +\frac{\mu}{4}\int_{0}^{\sigma(T)}\intt|\nabla u|^{2}+C(\bar{\rho})(\gamma-1)E_{0}+
 C\left((\gamma-1)^{\frac{1}{6}}E_{0}^{\frac{1}{2}}\right)^{\frac{2}{27}+1}.
\end{align}
It thus holds that
\begin{align}
  &\D\quad\frac{1}{2}\sup_{0\leq t\leq \sigma(T)}\|\rho^{\frac{1}{2}}u\|_{L^{2}}^{2}
  +\int_{0}^{\sigma(T)}\intt\left(\frac{\mu}{2}|\nabla u|^{2}+(\mu+\lambda)|\mdiv u|^{2}\right)\non\\
  &\D\leq C\left((\gamma-1)E_{0}\right)^{\frac{2}{3}}+C(\bar{\rho})(\gamma-1)E_{0}+
 C\left((\gamma-1)^{\frac{1}{6}}E_{0}^{\frac{1}{2}}\right)^{\frac{2}{27}+1}\non\\
 &\D\leq  CK_{3}(\gamma-1)^{\frac{1}{6}}E_{0}^{\frac{1}{2}}.
\end{align}
Here we have used the condition $(\gamma-1)^{\frac{1}{6}}E_{0}^{\frac{1}{2}}\leq 1$. We finish the proof of Lemma \ref{2dbu-L3.5}.
\end{proof}
\begin{lem}\label{2dbu-L3.6}
Under the same assumption as in Proposition \ref{2dbu-P3.1}, it holds that
\begin{align}\label{2dbu-E3.030}
  \D A_{1}(T)&\leq CK_{4}(\gamma-1)^{\frac{1}{6}}E_{0}^{\frac{1}{2}}
  +C\int_{0}^{\sigma(T)}\intt|\nabla u|^{2}\non\\
  &\D\quad +C\int_{0}^{T}\sigma\intt|\nabla u|^{3}
  +C\gamma\int_{0}^{T}\sigma\intt P|\nabla u|^{2}
\end{align}
and
\begin{align}\label{2dbu-E3.33}
  \D A_{2}(T)&\leq CK_{5}\left((\gamma-1)^{\frac{1}{6}}E_{0}^{\frac{1}{2}}\right)^{\frac{11}{18}
  }+C\gamma^{2}\int_{0}^{T}\intt\sigma^{2}|P\nabla u|^{2}\non\\
  &\D\quad+C\int_{0}^{T}\intt\sigma^{2}|\nabla u|^{4}+CA_{1}(T),
 \end{align}
 provided that
 \begin{align}
   \D (\gamma-1)^{\frac{1}{6}}E_{0}^{\frac{1}{2}}\leq \varepsilon_{2}\triangleq \min\left\{\varepsilon_{1},(4C_{1})^{-\frac{27}{2}}+(4C_{2})^{-27}\right\},
 \end{align}
 where $K_{4}$ and $K_{5}$ are given by
 \begin{align}
  \D K_{4}&=K_{2}^{5}M_{2}+K_{2}^{2}+K_{2}^{2}K_{1},\\
  \D K_{5}&=K_{1}+K_{2}^{2}(\gamma-1)^{\frac{2}{9}}+K_{2}^{3}M_{2}.
 \end{align}
\end{lem}
\begin{proof}
The basic idea of the proof of this lemma is due to Hoff \cite{hoff}, Huang-Li-Xin \cite{X.D. Huang}, Li-Xu-Zhang \cite{H.L. L} and Hou-Peng-Zhu \cite{X.F. Hou}.
 Multiplying $(\ref{2dbu-E1.1})_{2}$ by $\sigma \dot{u}$, and integrating over $\mathbb{R}^{3}$, one has
\begin{align}\label{2dbu-E3.032}
  \D\intt\sigma\rho|\dot{u}|^{2}&=-\intt\sigma \dot{u}\cdot\nabla P
  +\mu\intt\sigma\triangle u\cdot\dot{u}\non\\[2mm]
  &\D\quad+(\mu+\lambda)\intt\nabla\mdiv u\cdot\sigma\dot{u}
  +\intt\left(H\cdot\nabla H-\frac{1}{2}\nabla |H|^{2}\right)\cdot\sigma\dot{u}\non\\[2mm]
  &\D=\sum_{i=1}^{4}\mathscr{E}_{i}.
\end{align}
Moreover, using integration by parts, we have
\begin{align}\label{2dbu-E3.033}
 \D\mathscr{E}_{1}&=\intt\sigma \dot{u}\cdot\nabla P\non\\[2mm]
 \D&=\intt\sigma\mdiv u_{t}P+\intt\sigma\mdiv(u\cdot\nabla u)P\non\\[2mm]
 \D&=\rf\int\sigma\mdiv uP-\intt\sigma'P\mdiv u-\intt\sigma\mdiv u P_{t}
 +\intt\sigma\mdiv(u\cdot\nabla u)P\non\\[2mm]
 \D&=\rf\int\sigma\mdiv uP-\sigma'\intt P\mdiv u
 +(\gamma-1)\sigma\intt P|\mdiv u|^{2}+\intt\sigma P|\nabla u|^{2},
\end{align}
\begin{align}
 \D\mathscr{E}_{2}&=\intt\mu\sigma\triangle u\cdot \dot{u}\non\\[2mm]
 \D&=-\mu\intt\sigma\nabla u\cdot\nabla u_{t}
 +\mu\intt\sigma\triangle u\cdot(u\cdot\nabla u)\non\\[2mm]
 \D&\leq-\left(\frac{\mu}{2}\int\sigma|\nabla u|^{2}\right)_{t}
 -\frac{\mu\sigma'}{2}\intt|\nabla u|^{2}+2\mu\sigma\intt|\nabla u|^{3}
\end{align}
and
\begin{align}
  \D\mathscr{E}_{3}&=\intt(\mu+\lambda)\sigma\nabla\mdiv u\cdot \dot{u}\non\\[2mm]
  \D&\leq-\frac{\mu+\lambda}{2}\rf\intt\sigma|\mdiv u|^{2}
  +\frac{\mu+\lambda}{2}\sigma'\intt|\mdiv u|^{2}
  +2(\mu+\lambda)\sigma\intt|\nabla u|^{3}.
\end{align}
It remains to estimate $\mathscr{E}_{4}$. In fact,
\begin{align}\label{2dbu-E3.036}
 \D\quad\mathscr{E}_{4}&=\intt\left(H\cdot\nabla H
 -\frac{1}{2}\nabla |H|^{2}\right)\cdot\left(\sigma u_{t}+\sigma u\cdot\nabla u\right)\non\\[2mm]
 &\D=\rf\intt\left(H\cdot\nabla H-\frac{1}{2}\nabla |H|^{2}\right)\cdot\sigma u
 -\intt\left(H\cdot\nabla H-\frac{1}{2}\nabla |H|^{2}\right)_{t}\cdot\sigma u\non\\[2mm]
 &\quad-\intt\left(H\cdot\nabla H-\frac{1}{2}\nabla |H|^{2}\right)\cdot\sigma'u
 +\intt\sigma\left(H\cdot\nabla H-\frac{1}{2}\nabla |H|^{2}\right)\cdot(u\cdot\nabla u)\non\\[2mm]
 &\D\leq \rf\intt\left(H\cdot\nabla H-\frac{1}{2}\nabla |H|^{2}\right)\cdot\sigma u
 +C\sigma\|H\|_{L^{6}}\|H_{t}\|_{L^{2}}\|\nabla u\|_{L^{3}}\non\\[2mm]
 &\D\quad+\sigma'\|\nabla u\|_{L^{2}}\|H\|_{L^{4}}^{2}+C\sigma\|H\|_{L^{6}}\|\nabla H\|_{L^{6}}\|u\|_{L^{6}}\|\nabla u\|_{L^{2}}\non\\[2mm]
 &\D\leq \rf\intt\left(H\cdot\nabla H-\frac{1}{2}\nabla |H|^{2}\right)\cdot\sigma u
 +C\sigma\|\nabla H\|_{L^{2}}^{6}+\sigma\|H_{t}\|_{L^{2}}^{2}+\sigma\|\nabla u\|_{L^{3}}^{3}\non\\[2mm]
 &\D\quad+C\sigma'\|\nabla u\|_{L^{2}}^{2}
  +C\sigma'\|\nabla H\|_{L^{2}}^{2}\|H\|_{L^{3}}^{2}
  +C\sigma\|\nabla^{2} H\|_{L^{2}}^{2}+\sigma\|\nabla u\|_{L^{2}}^{4}\|\nabla H\|_{L^{2}}^{2}.
\end{align}
Substituting (\ref{2dbu-E3.033})-(\ref{2dbu-E3.036}) into (\ref{2dbu-E3.032}), we have
\begin{align}\label{2dbu-E3.037}
  &\quad\D \rf\Big\{\intt\left(\frac{\mu}{2}\sigma\|\nabla u\|_{L^{2}}^{2}
  +\frac{\mu+\lambda}{2}\sigma\|\mdiv u\|_{L^{2}}^{2}\right)\non\\[2mm]
  &\quad\ +\sigma\intt\mdiv uP
   -\intt\left(H\cdot\nabla H-\frac{1}{2}\nabla |H|^{2}\right)\cdot\sigma u\Big\}
   +\intt\sigma\rho|\dot{u}|^{2}\non\\[2mm]
  &\D\leq \frac{\mu\sigma'}{2}\intt|\nabla u|^{2}+2\mu\sigma\intt|\nabla u|^{3}
  +\frac{\mu+\lambda}{2}\sigma'\intt|\mdiv u|^{2}
  +2(\mu+\lambda)\sigma\intt|\nabla u|^{3}\non\\[2mm]
  &\D\quad+C\sigma\|\nabla H\|_{L^{2}}^{6}+\sigma\|H_{t}\|_{L^{2}}^{2}
  +\sigma\|\nabla u\|_{L^{3}}^{3}
  -\sigma'\intt P\mdiv u+\gamma\sigma\intt P|\nabla u|^{2}\non\\[2mm]
  &\D\quad+C\sigma'\|\nabla u\|_{L^{2}}^{2}
  +C\sigma'\|\nabla H\|_{L^{2}}^{2}\|H\|_{L^{3}}^{2}
  +C\sigma\|\nabla^{2} H\|_{L^{2}}^{2}
  +\sigma\|\nabla u\|_{L^{2}}^{4}\|\nabla H\|_{L^{2}}^{2}.
\end{align}
Integrate (\ref{2dbu-E3.037}) over $(0,T)$, we have
\begin{align}
  &\quad\D \sup_{0\leq t\leq T}\left(\frac{\mu}{2}\sigma\|\nabla u\|_{L^{2}}^{2}
  +\frac{\mu+\lambda}{2}\sigma\|\mdiv u\|_{L^{2}}^{2}\right)
  +\int_{0}^{T}\sigma\|\rho^{\frac{1}{2}}\dot{u}\|_{L^{2}}^{2}\non\\[2mm]
   &\D\leq\sigma\intt\left(H\cdot\nabla H
   -\frac{1}{2}\nabla |H|^{2}\right)\cdot u
   +C\int_{0}^{\sigma(T)}\intt|\nabla u|^{2}
   +C(4\mu+\lambda+1)\int_{0}^{T}\sigma\intt|\nabla u|^{3}\non\\[2mm]
  &\D\quad+\int_{0}^{\sigma(T)}\intt |P\mdiv u|
  +\gamma\int_{0}^{T}\sigma\intt P|\nabla u|^{2}
  +\int_{0}^{T}\sigma\|\nabla H\|_{L^{2}}^{6}+\int_{0}^{T}\sigma\|H_{t}\|_{L^{2}}^{2}\non\\[2mm]
  &\D\quad+\int_{0}^{\sigma(T)}\|\nabla H\|_{L^{2}}^{2}\|H\|_{L^{3}}^{2}
  +C(\gamma-1)E_{0}+\int_{0}^{T}\sigma\|\nabla^{2}H\|_{L^{2}}^{2}
  +\int_{0}^{T}\sigma\|\nabla H\|_{L^{2}}^{2}\|\nabla u\|_{L^{2}}^{4}\non\\[2mm]
  &\D\leq\epsilon\mu\sigma\|\nabla u\|_{L^{2}}^{2}
  +\frac{C(\epsilon)}{\mu}\sigma\|H\|_{L^{4}}^{4}+C\int_{0}^{\sigma(T)}\intt|\nabla u|^{2}
  +C\int_{0}^{T}\sigma\intt|\nabla u|^{3}\non\\[2mm]
  &\D\quad+C\gamma\int_{0}^{T}\sigma\intt P|\nabla u|^{2}+C(\bar{\rho},\epsilon)(\gamma-1)E_{0}
  +CK_{2}^{5}M_{2}\left((\gamma-1)^{\frac{1}{6}}E_{0}^{\frac{1}{2}}\right)^{2}\non\\[2mm]
  &\D\quad+CK_{2}^{2}(\gamma-1)^{\frac{1}{6}}E_{0}^{\frac{1}{2}}
      +K_{2}^{2}\left((\gamma-1)^{\frac{1}{6}}E_{0}^{\frac{1}{2}}\right)^{1+\frac{2}{27}}
      +K_{2}^{2}K_{1}\left((\gamma-1)^{\frac{1}{6}}E_{0}^{\frac{1}{2}}\right)^{1+\frac{1}{9}},
\end{align}
which leads to
\begin{align}\label{2dbu-E3.039}
  &\quad\D \sup_{0\leq t\leq T}\left(\frac{\mu}{4}\sigma\|\nabla u\|_{L^{2}}^{2}
  +\frac{\mu+\lambda}{2}\sigma\|\mdiv u\|_{L^{2}}^{2}\right)
  +\int_{0}^{T}\sigma\|\rho^{\frac{1}{2}}\dot{u}\|_{L^{2}}^{2}dt\non\\
  &\D\leq CK_{4}(\gamma-1)^{\frac{1}{6}}E_{0}^{\frac{1}{2}}
  +C\int_{0}^{\sigma(T)}\intt|\nabla u|^{2}
  +C\int_{0}^{T}\sigma\intt|\nabla u|^{3}\non\\
  &\D\quad+C\gamma\int_{0}^{T}\sigma\intt P|\nabla u|^{2}.
\end{align}
Then, by the virtue of (\ref{2dbu-E3.08}), we get (\ref{2dbu-E3.030}).

Next, operating $\partial_{t}+\mdiv(u\cdot)$ to the both sides of the $j$th equation of $(\ref{2dbu-E1.1})_{2}$, yields that
\begin{align}\label{2dbu-E3.040}
  &\D\quad(\rho\dot{u}^{j})_{t}+\mdiv(\rho u\dot{u}^{j})-\mu\Delta \dot{u}^{j}
  -(\mu+\lambda)\partial_{j}\mdiv \dot{u}\non\\[2mm]
  &\D=\mu\partial_{i}(-\partial_{i}u\cdot\nabla u^{j}
  +\div u\partial_{i}u^{j})-\mu\mdiv(\partial_{i}u\partial_{i}u^{j})
  -(\mu+\lambda)\partial_{j}\left[\partial_{i}u\cdot\nabla u^{i}
  -|\mdiv u|^{2}\right]\non\\[2mm]
  &\D\quad-\mdiv(\partial_{j}u(\mu+\lambda)\mdiv u)
  +(\gamma-1)\partial_{j}(P\mdiv u)+\mdiv(P\partial_{j}u)\non\\[2mm]
  &\D\quad +\left[(H\cdot\nabla H^{j})_{t}+\mdiv(uH\cdot\nabla H^{j})\right]
  -\frac{1}{2}\left[(\partial_{j}|H|^{2})_{t}+\mdiv(u\partial_{j}|H|^{2})\right].
\end{align}
Multiplying (\ref{2dbu-E3.040}) by $\sigma^{m}\dot{u}^{j}$ for $m\geq 0$, and integrating by parts over $\mathbb{R}^{3}$, we obtain after summing them with respect to $j$ that
\begin{align}
 &\D \quad\frac{1}{2}\rf\intt\sigma^{m}\rho|\dot{u}|^{2}+\mu\intt\sigma^{m}|\nabla\dot{u}|^{2}
 +(\mu+\lambda)\intt\sigma^{m}|\mdiv\dot{u}|^{2}\non\\[2mm]
 &\D=\frac{m}{2}\sigma^{m-1}\sigma'\intt\rho|\dot{u}|^{2}
 +\mu\intt\sigma^{m}\partial_{i}\dot{u}^{j}(\partial_{i}u\cdot\nabla u^{i}
 -\div u\partial_{i}u^{j})\non\\[2mm]
 &\D\quad+\mu\intt\sigma^{m}\partial_{k}\dot{u}^{j}\partial_{i}u^{k}\partial_{i}u^{j}+
 \intt\sigma^{m}\partial_{j}\dot{u}^{j}\left[(\mu+\lambda)\partial_{i}u\cdot\nabla u^{i}
 -\mu|\mdiv u|^{2}\right]\non\\[2mm]
 &\D\quad+(\mu+\lambda)\intt\sigma^{m}\partial_{k}\dot{u}^{j}\partial_{j}u^{k}\mdiv u
 -\intt\sigma^{m}\partial_{k}\dot{u}^{j}\partial_{j}u^{k}P\non\\[2mm]
 &\D\quad-(\gamma-1)\intt\sigma^{m} \partial_{j}\dot{u}^{j}\partial_{k}u^{k}P
 -\frac{1}{2}\intt\sigma^{m}\dot{u}^{j}
 \left[(\partial_{j}|H|^{2})_{t}+\mdiv(u\partial_{j}|H|^{2})\right]\non\\[2mm]
 &\D\quad+\intt\sigma^{m}\dot{u}^{j}\left[(H\cdot\nabla H^{j})_{t}
 +\mdiv(uH\cdot\nabla H^{j})\right]\non\\[2mm]
 &\D\leq\frac{m}{2}\sigma^{m-1}\sigma'\intt\rho|\dot{u}|^{2}
 +\frac{\mu}{2}\intt\sigma^{m}|\nabla\dot{u}|^{2}
 +\frac{\mu+\lambda}{2}\intt\sigma^{m}|\mdiv\dot{u}|^{2}\non\\[2mm]
 &\D\quad+C\mu\intt\sigma^{m}|\nabla u|^{4}+C(\mu+\lambda)
 \left(1+\frac{\lambda}{\mu}\right)\intt\sigma^{m}|\nabla u|^{4}
 +C\gamma^{2}\intt\sigma^{m}|P\nabla u|^{2}\non\\[2mm]
 &\D\quad+C_{1}\sigma^{m}\|H\|_{L^{3}}^{2}\|\nabla H_{t}\|_{L^{2}}^{2}
 +C\sigma^{m}\left(\|\nabla u\|_{L^{2}}^{4}
 +\|\nabla H\|_{L^{2}}^{4}\right)\|\nabla^{2}H\|_{L^{2}}^{2}\non\\[2mm]
 &\D\leq\frac{m}{2}\sigma^{m-1}\sigma'\intt\rho|\dot{u}|^{2}
 +\frac{\mu}{2}\intt\sigma^{m}|\nabla\dot{u}|^{2}
  +\frac{\mu+\lambda}{2}\intt\sigma^{m}|\mdiv\dot{u}|^{2}\non\\[2mm]
  &\D\quad+C\mu\intt\sigma^{m}|\nabla u|^{4}+C\gamma^{2}\intt\sigma^{m}|P\nabla u|^{2}+C_{1}
  \sigma^{m}\left((\gamma-1)^{\frac{1}{6}}E_{0}^{\frac{1}{2}}\right)^{\frac{2}{27}}
  \|\nabla H_{t}\|_{L^{2}}^{2}\non\\[2mm]
  &\D\quad+C\sigma^{m}\left(\|\nabla u\|_{L^{2}}^{4}
  +\|\nabla H\|_{L^{2}}^{4}\right)\|\nabla^{2}H\|_{L^{2}}^{2},
 \end{align}
 which implies that
 \begin{align}\label{2dbu-E3.042}
   &\D \quad\frac{1}{2}\rf\intt\sigma^{m}\rho|\dot{u}|^{2}
   +\intt\sigma^{m}|\nabla\dot{u}|^{2}
   -C_{1}\sigma^{m}\left((\gamma-1)^{\frac{1}{6}}E_{0}^{\frac{1}{2}}\right)^{\frac{2}{27}}
   \|\nabla H_{t}\|_{L^{2}}^{2}\non\\
   &\D\leq C\sigma^{m-1}\sigma'\intt\rho|\dot{u}|^{2}+C\mu\intt\sigma^{m}|\nabla u|^{4}
   +C\gamma^{2}\intt\sigma^{m}|P\nabla u|^{2}\non\\[2mm]
  &\D\quad+C\sigma^{m}\left(\|\nabla u\|_{L^{2}}^{4}
  +\|\nabla H\|_{L^{2}}^{4}\right)\|\nabla^{2}H\|_{L^{2}}^{2}.
 \end{align}
  Similar to \cite{H.L. L}, noting that
 \begin{align}
  \D H_{tt}-\nu\Delta H_{t}=(H\cdot\nabla u-u\cdot\nabla H-H\mdiv u)_{t},
 \end{align}
 and using the identity $\D u_{t}=\dot{u}-u\cdot\nabla u$, we obtain by direct computation that
 \begin{align}\label{2dbu-E3.044}
  &\D\quad\frac{1}{2}\rf\intt\sigma^{m}|H_{t}|^{2}+\nu\intt\sigma^{m}|\nabla H_{t}|^{2}
  -\frac{m}{2}\sigma^{m-1}\sigma'\intt|H_{t}|^{2}\non\\
  &\D=\intt\sigma^{m}(H_{t}\cdot\nabla u-u\cdot\nabla H_{t}-H_{t}\mdiv u)\cdot H_{t}\non\\
  &\D\quad+\intt\sigma^{m}(H\cdot\nabla \dot{u}-\dot{u}\cdot \nabla H-H\mdiv\dot{u})\cdot H_{t}\non\\
  &\D\quad-\intt\sigma^{m}\left[H\cdot\nabla(u\cdot\nabla u)-(u\cdot\nabla u)\cdot\nabla H
  -H\mdiv(u\cdot\nabla u)\right]\cdot H_{t}\non\\
  &\D\triangleq\sum_{i=1}^{3}\mathscr{N}_{i}.
 \end{align}
 It follows from Cauchy inequality, Sobolev inequality and (\ref{2dbu-E3.02}) that
 \begin{align}\label{2dbu-E3.045}
   \D \mathscr{N}_{1}&\leq \epsilon\nu\sigma^{m}\|\nabla H_{t}\|_{L^{2}}^{2}
   +C\sigma^{m}\|H_{t}\|_{L^{2}}^{2}\|\nabla u\|_{L^{2}}^{4},\non\\[2mm]
   \D\mathscr{N}_{2}&\leq C\|H\|_{L^{3}}\sigma^{m}(\|\nabla \dot{u}\|_{L^{2}}^{2}
   +\|\nabla H_{t}\|_{L^{2}}^{2})\non\\[2mm]
   \D&\leq C_{2}\left((\gamma-1)^{\frac{1}{6}}E_{0}^{\frac{1}{2}}\right)^{\frac{1}{27}}
   \sigma^{m}(\|\nabla\dot{u}\|_{L^{2}}^{2}+\|\nabla H_{t}\|_{L^{2}}^{2}),\non\\[2mm]
   \D\mathscr{N}_{3}&\leq C\sigma^{m}\|\nabla u\|_{L^{2}}\|\nabla u\|_{L^{6}}\|\nabla H\|_{L^{2}}\|\nabla H_{t}\|_{L^{2}}.
 \end{align}
 Thanks to (\ref{2dbu-E2.06}), (\ref{2dbu-E3.03}) and Sobolev inequality, we deduce that
  \begin{align}\label{2dbu-E3.046}
   \D\|\nabla u\|_{L^{6}}&\leq C(\|\rho\dot{u}\|_{L^{2}}+\|P\|_{L^{6}}+\|H\cdot\nabla H\|_{L^{2}})\non\\
   \D&\leq C(\bar{\rho})\left(\|\rho^{\frac{1}{2}}\dot{u}\|_{L^{2}}+\|H\|_{L^{3}}\|\nabla^{2}H\|_{L^{2}}\right)
   +C(\bar{\rho})(\gamma-1)^{\frac{1}{6}}E_{0}^{\frac{1}{6}}.
 \end{align}
 Combining (\ref{2dbu-E3.045}) and (\ref{2dbu-E3.046}), we get
 \begin{align}\label{2dbu-E3.047}
   \D\mathscr{N}_{3}&\leq \frac{\nu}{8}\sigma^{m}\|\nabla H_{t}\|_{L^{2}}^{2}
   +C\sigma^{m}\left(\|\nabla u\|_{L^{2}}^{4}+\|\nabla H\|_{L^{2}}^{4}\right)\|H\|_{L^{3}}^{2}\|\nabla^{2}H\|_{L^{2}}^{2}\non\\
   \D&+C\sigma^{m}\|\nabla u\|_{L^{2}}^{2}\|\rho^{\frac{1}{2}}\dot{u}\|_{L^{2}}^{2}\|\nabla H\|_{L^{2}}^{2}
   +C\sigma^{m}\|\nabla u\|_{L^{2}}^{2}\|\nabla H\|_{L^{2}}^{2}(\gamma-1)^{\frac{1}{3}}E_{0}^{\frac{1}{3}}.
 \end{align}
Consequently, substituting (\ref{2dbu-E3.045})-(\ref{2dbu-E3.047}) into (\ref{2dbu-E3.044}) yields that
 \begin{align}\label{2dbu-E3.048}
   &\D\quad\rf\intt\sigma^{m}|H_{t}|^{2}+\intt\sigma^{m}|\nabla H_{t}|^{2}-C_{2}\nu\sigma^{m}\left((\gamma-1)^{\frac{1}{6}}
   E_{0}^{\frac{1}{2}}\right)^{\frac{1}{27}}(\|\nabla \dot{u}\|_{L^{2}}^{2}
   +\|\nabla H_{t}\|_{L^{2}}^{2})\non\\[2mm]
   &\D\leq \frac{m}{2}C\sigma^{m-1}\sigma'\intt|H_{t}|^{2}
   +C\sigma^{m}\|\nabla u\|_{L^{2}}^{2}\|\rho^{\frac{1}{2}}\dot{u}\|_{L^{2}}^{2}\|\nabla H\|_{L^{2}}^{2}\non\\[2mm]
   &\D\quad+C\sigma^{m}\left(\|\nabla u\|_{L^{2}}^{4}+\|\nabla H\|_{L^{2}}^{4}\right)\left(\|H\|_{L^{3}}^{2}\|\nabla^{2}H\|_{L^{2}}^{2}
   +\|H_{t}\|_{L^{2}}^{2}\right)\non\\[2mm]
   &\D\quad+C\sigma^{m}\|\nabla u\|_{L^{2}}^{2}\|\nabla H\|_{L^{2}}^{2}(\gamma-1)^{\frac{1}{3}}E_{0}^{\frac{1}{3}}.
 \end{align}
 Thus, taking $m=2$ in (\ref{2dbu-E3.042}) and (\ref{2dbu-E3.048}), and integrating the resulting equation over $(0, T)$, we deduce that
 \begin{align}\label{2dbu-E3.051}
  &\D\quad\sup_{0\leq t\leq T}\sigma^{2}\left(\|\rho^{\frac{1}{2}}\dot{u}\|_{L^{2}}^{2}
  +\|H_{t}\|_{L^{2}}^{2}\right)
  +\int_{0}^{T}\sigma^{2}\left(\|\nabla\dot{u}\|_{L^{2}}^{2}
  +\|\nabla H_{t}\|_{L^{2}}^{2}\right)\non\\
  &\D\leq CK_{5}\left((\gamma-1)^{\frac{1}{6}}E_{0}^{\frac{1}{2}}\right)^{\frac{11}{18}}
  +C\gamma^{2}\int_{0}^{T}\intt\sigma^{2}|P\nabla u|^{2}\non\\
  &\D\quad+C\int_{0}^{T}\intt\sigma^{2}|\nabla u|^{4}+CA_{1}(T),
 \end{align}
 provided that
 \begin{align}
   \D (\gamma-1)^{\frac{1}{6}}E_{0}^{\frac{1}{2}}\leq \varepsilon_{2}\triangleq \min\left\{\varepsilon_{1},(4C_{1})^{-\frac{27}{2}}+(4C_{2})^{-27}\right\},
 \end{align}
 where we have used the following inequalities:
 \begin{align}
   &\D \quad \int_{0}^{T}\sigma^{2}\|\nabla u\|_{L^{2}}^{2}
   \|\rho^{\frac{1}{2}}\dot{u}\|_{L^{2}}^{2}\|\nabla H\|_{L^{2}}^{2}\non\\
  &\D\leq \sup_{0\leq t\leq T}\left(\sigma^{2}\|\rho^{\frac{1}{2}}\dot{u}\|_{L^{2}}^{2}\right)
  \int_{0}^{T}\left(\|\nabla u\|_{L^{2}}^{4}+\|\nabla H\|_{L^{2}}^{4}\right)\non\\
  &\D\leq CK_{1}\left((\gamma-1)^{\frac{1}{6}}E_{0}^{\frac{1}{2}}\right)^{\frac{11}{18}}+
  CK_{2}^{3}M_{2}\left((\gamma-1)^{\frac{1}{6}}E_{0}^{\frac{1}{2}}\right)^{\frac{3}{2}},\label{2dbu-E3.0057}\\
   &\D\quad \int_{0}^{T}\sigma^{2}\|\nabla u\|_{L^{2}}^{2}\|\nabla H\|_{L^{2}}^{2}
  (\gamma-1)^{\frac{1}{3}}E_{0}^{\frac{1}{3}}\non\\
  &\D\leq C\sup_{0\leq t\leq T}\left(\sigma\|\nabla u\|_{L^{2}}^{2}\right)(\gamma-1)^{\frac{1}{3}}
  E_{0}^{\frac{1}{3}}\int_{0}^{T}\|\nabla H\|_{L^{2}}^{2}\non\\
  &\D\leq CK_{2}^{2}\left((\gamma-1)^{\frac{1}{6}}E_{0}^{\frac{1}{2}}\right)^{\frac{3}{2}
 }(\gamma-1)^{\frac{1}{3}}
  E_{0}^{\frac{1}{3}}\non\\
  &\D\leq CK_{2}^{2}(\gamma-1)^{\frac{23}{72}}\left((\gamma-1)^{\frac{1}{6}}E_{0}^{\frac{1}{2}}\right)^{\frac{3}{2}},\\
   &\D\quad\int_{0}^{T}\sigma^{2}\left(\|\nabla u\|_{L^{2}}^{4}+\|\nabla H\|_{L^{2}}^{4}\right)\left(\|H\|_{L^{3}}^{2}\|\nabla^{2}H\|_{L^{2}}^{2}+\|H_{t}\|_{L^{2}}^{2}\right)\non\\
   &\D\leq C \sup_{0\leq t\leq T}\Big(\sigma^{2}\|\nabla^{2}H\|_{L^{2}}^{2}\Big)\Big(\sup_{0\leq t\leq T}\|H\|_{L^{3}}^{3}\Big)^{\frac{2}{3}}\int_{0}^{T}\|\nabla u\|_{L^{2}}^{4}\non\\
   &\D\quad+ \sup_{0\leq t\leq T}\Big(\sigma^{2}\|H_{t}\|_{L^{2}}^{2}\Big)\int_{0}^{T}\|\nabla H\|_{L^{2}}^{4}\non\\
   &\D\leq CK_{1}\left((\gamma-1)^{\frac{1}{6}}E_{0}^{\frac{1}{2}}\right)^{\frac{37}{54}}
   +CK_{2}^{3}M_{2}\left((\gamma-1)^{\frac{1}{6}}E_{0}^{\frac{1}{2}}\right)^{\frac{3}{2}}\non\\
   &\D\leq C(K_{1}+K_{2}^{3}M_{2})\left((\gamma-1)^{\frac{1}{6}}E_{0}^{\frac{1}{2}}\right)^{\frac{11}{18}}.\label{2dbu-E3.0059}
 \end{align}
 Here, to obtained (\ref{2dbu-E3.0057})-(\ref{2dbu-E3.0059}), (\ref{2dbu-E3.08}) and (\ref{2dbu-E3.9}) have been used.

 Furthermore, it holds from $(\ref{2dbu-E1.1})_{3}$ that
 \begin{align}
  \D \Delta H=H_{t}-(H\cdot\nabla u-u\cdot\nabla H-\mdiv u H),
 \end{align}
 then standard $L^{p}$ estimate for elliptic equation gives
 \begin{align}\label{2dbu-E3.57}
   \D \|\nabla^{2}H\|_{L^{2}}&\leq C\left(\|H_{t}\|_{L^{2}}
   +\|H\cdot\nabla u-u\cdot\nabla H-\mdiv u H\|_{L^{2}}\right)\non\\
 \D&\leq C\left(\|H_{t}\|_{L^{2}}+\|\nabla H\|_{L^{2}}^{\frac{1}{2}}\|\nabla^{2}H\|_{L^{2}}^{\frac{1}{2}}\|\nabla u\|_{L^{2}}\right).
 \end{align}
By (\ref{2dbu-E3.01}) and (\ref{2dbu-E3.051}), we have
\begin{align}\label{2dbu-E3.059}
   &\D\quad \sup_{0\leq t\leq T}\left(\sigma^{2}\|\nabla^{2}H\|_{L^{2}}^{2}\right)\non\\[2mm]
   &\D\leq C\sup_{0\leq t\leq T}\sigma^{2}\left(\|\nabla H\|_{L^{2}}^{2}\|\nabla u\|_{L^{2}}^{4}
   +\|H_{t}\|_{L^{2}}^{2}\right)\non\\[2mm]
   &\D\leq C\sup_{0\leq t\leq \sigma(T)}\sigma^{2}\left(\|\nabla H\|_{L^{2}}^{2}\|\nabla u\|_{L^{2}}^{4}\right)
   +C\sup_{\sigma(T)\leq t\leq T}\sigma^{2}\left(\|\nabla H\|_{L^{2}}^{2}\|\nabla u\|_{L^{2}}^{4}\right)\non\\[2mm]
   &\D\quad+C\sup_{0\leq t\leq T}\sigma^{2}\|H_{t}\|_{L^{2}}^{2}\non\\[2mm]
   &\D\leq C\sup_{0\leq t\leq \sigma(T)}
   \left((\sigma\|\nabla H\|_{L^{2}}^{2})\Big(\sigma^{\frac{1}{3}}\|\nabla u\|_{L^{2}}^{2}\Big)^{2}\right)
   +C\sup_{\sigma(T)\leq t\leq T}\left(\Big(\sigma\|\nabla H\|_{L^{2}}^{2}\Big)
   \Big(\sigma\|\nabla u\|_{L^{2}}^{2}\Big)^{2}\right)\non\\[2mm]
   &\D\quad +CK_{5}\left((\gamma-1)^{\frac{1}{6}}E_{0}^{\frac{1}{2}}\right)^{\frac{11}{18}}
   +C\gamma^{2}\int_{0}^{T}\intt\sigma^{2}|P\nabla u|^{2}\non\\
  &\D\quad+C\int_{0}^{T}\intt\sigma^{2}|\nabla u|^{4}+CA_{1}(T)\non\\[2mm]
  &\D\leq CK_{5}\left((\gamma-1)^{\frac{1}{6}}E_{0}^{\frac{1}{2}}\right)^{\frac{11}{18}}
  +C\gamma^{2}\int_{0}^{T}\intt\sigma^{2}|P\nabla u|^{2}\non\\
  &\D\quad+C\int_{0}^{T}\intt\sigma^{2}|\nabla u|^{4}+CA_{1}(T).
 \end{align}
 Due to (\ref{2dbu-E3.08}), (\ref{2dbu-E3.051}) and (\ref{2dbu-E3.059}), we finally get (\ref{2dbu-E3.33}). And Lemma \ref{2dbu-L3.6} is proved.
\end{proof}
The following lemma will play a crucial role in the proof of the upper bound of the density.
\begin{lem}\label{2dbu-L3.7}
Under the same assumption as in Proposition \ref{2dbu-P3.1}, it holds that
\begin{align}
   \D& \sup_{0\leq t\leq \sigma(T)}\Big(\|\nabla u\|_{L^{2}}^{2}+\|\mdiv u\|_{L^{2}}^{2}\Big)
   +\int_{0}^{\sigma(T)}\intt\rho|\dot{u}|^{2}\leq K_{6},\label{2dbu-E3.0061}\\
   \D &\sup_{0\leq t\leq \sigma(T)} t\bi(\|\rho^{\frac{1}{2}}\dot{u}\|_{L^{2}}^{2}+\|H_{t}\|_{L^{2}}^{2}\bi)
  +\int_{0}^{\sigma(T)} t\bi(\|\nabla\dot{u}\|_{L^{2}}^{2}+\|\nabla H_{t}\|_{L^{2}}^{2}\bi)\leq K_{7},\label{2dbu-E3.0062}
 \end{align}
and
\begin{align}
\D&\sup_{0\leq t\leq T}\|\nabla u\|_{L^{2}}^{2}+\int_{0}^{T}\intt\rho|\dot{u}|^{2}\leq C(K_{6}+1),\label{2dbu-E3.0063}\\
    &\D \quad\sup_{0\leq t\leq T} \sigma\bi(\|\rho^{\frac{1}{2}}\dot{u}\|_{L^{2}}^{2}+\|H_{t}\|_{L^{2}}^{2}\bi)
  +\int_{0}^{T} \sigma\bi(\|\nabla\dot{u}\|_{L^{2}}^{2}+\|\nabla H_{t}\|_{L^{2}}^{2}\bi)\leq C(K_{7}+1),\label{2dbu-E3.0064}
 \end{align}
 provided that
 \begin{align}
   \D (\gamma-1)^{\frac{1}{6}}E_{0}^{\frac{1}{2}}\leq \varepsilon_{3}\triangleq \min\left\{\varepsilon_{1},\varepsilon_{2},4C_{3}(\bar{\rho}))^{-27},1\right\},
 \end{align}
 where
 \begin{align}
  \D K_{6}&=C(\bar{\rho})(K_{2}M_{2}+K_{3}+K_{2}^{2}+M_{2}^{\frac{3}{4}}M_{1}^{\frac{1}{2}}+1)
  +C(\bar{\rho})(\gamma-1)^{\frac{2}{9}}+CK_{3}(\gamma-1)^{\frac{4}{9}}
   \non\\
   &\D\quad+CK_{2}M_{2}(\gamma-1)^{\frac{4}{9}}+CK_{3}(\gamma-1)^{\frac{1}{3}}
   +(\gamma-1)^{\frac{1}{3}}+C(K_{3}+1)(\gamma-1),\\
   \D K_{7}&=2\max\{K_{7}^{1}, K_{7}^{2}\},\\
   \D K_{7}^{1}&=CK_{6}+CK_{6}^{\frac{3}{2}}+CK_{6}^{\frac{1}{2}}K_{2}M_{2}+CK_{6}^{\frac{1}{2}}(\gamma-1)^{\frac{1}{3}}
  +C(K_{6}^{2}+K_{2}^{2}M_{2}^{2})\non\\
  &\D\quad+C(\gamma-1)^{\frac{1}{3}}(K_{6}^{\frac{3}{2}}+K_{6}^{\frac{1}{2}}K_{2}M_{2}+K_{6}^{\frac{1}{2}}(\gamma-1)^{\frac{1}{3}})^{\frac{1}{2}},\non\\
  \D K_{7}^{2}&=CK_{2}M_{2}+CK_{6}^{2}K_{2}^{2}+CK_{2}^{4}M_{2}^{2}+C(K_{6}^{2}+CK_{2}^{2}M_{2}^{2})K_{2}^{2}+CK_{6}K_{2}^{2}(\gamma-1)^{\frac{2}{9}}.
 \end{align}
\end{lem}
\begin{proof}
Multiplying $(\ref{2dbu-E1.1})_{2}$ by $u_{t}$,  we have
 \begin{align}\label{2dbu-E3.63}
  \D&\quad \rho|\dot{u}|^{2}+\frac{1}{2}\Big(\mu|\nabla u|^{2}
  +(\mu+\lambda)|\mdiv u|^{2}\Big)_{t}-\left(P\mdiv u+\Big(H\cdot\nabla H-\frac{1}{2}\nabla|H|^{2}\Big)\cdot u\right)_{t}\non\\
  \D&=-\Big(H\cdot\nabla H-\frac{1}{2}\nabla|H|^{2}\Big)_{t}\cdot u
  -P_{t}\mdiv u+\rho u\cdot\nabla u\cdot\dot{u}.
 \end{align}
 Integrating (\ref{2dbu-E3.63}) over $\mathbb{R}^{3}$, one has
 \begin{align}\label{2dbu-E3.64}
  & \D\quad\frac{1}{2}\rf\intt\Big(\mu|\nabla u|^{2}+(\mu+\lambda)|\mdiv u|^{2}-P\mdiv u
  -\Big(H\cdot\nabla H-\frac{1}{2}\nabla|H|^{2}\Big)\Big)\cdot u+\intt\rho|\dot{u}|^{2}\non\\[2mm]
  &\D=\intt \rho\dot{u}\cdot (u\cdot\nabla u)
  -\intt\Big(H\cdot\nabla H-\frac{1}{2}\nabla|H|^{2}\Big)_{t}\cdot u-\intt P_{t}\mdiv u\non\\[1mm]
  &\D\leq C(\bar{\rho})\|\rho^{\frac{1}{2}}\dot{u}\|_{L^{2}}
  \|\rho^{\frac{1}{3}}u\|_{L^{3}}\|\nabla u\|_{L^{6}}+\intt \mdiv u\mdiv (Pu)
  +(\gamma-1)\intt P|\mdiv u|^{2}\non\\[2mm]
  &\D\quad+ \|H\|_{L^{3}}\|H_{t}\|_{L^{2}}\|\nabla u\|_{L^{6}}\non\\
  &\D\leq C(\bar{\rho})\|\rho^{\frac{1}{2}}\dot{u}\|_{L^{2}}\|\rho^{\frac{1}{3}}u\|_{L^{3}}
  \left(\|\rho^{\frac{1}{2}}\dot{u}\|_{L^{2}}+\|P\|_{L^{6}}+\|H\cdot\nabla H\|_{L^{2}}\right)
  -\intt Pu\nabla\mdiv u\non\\[2mm]
  &\D\quad+C(\bar{\rho})(\gamma-1)\|\nabla u\|_{L^{2}}^{2}
  +\|H\|_{L^{3}}\|H_{t}\|_{L^{2}}
   \left(\|\rho^{\frac{1}{2}}\dot{u}\|_{L^{2}}+\|P\|_{L^{6}}
   +\|H\cdot\nabla H\|_{L^{2}}\right)\non\\[2mm]
   &\D\leq C(\bar{\rho})\|\rho^{\frac{1}{2}}\dot{u}\|_{L^{2}}^{2}A_{5}^{\frac{1}{3}}(\sigma(T))
   + C(\bar{\rho})\|\rho^{\frac{1}{2}}\dot{u}\|_{L^{2}}A_{5}^{\frac{1}{3}}(\sigma(T))\|P\|_{L^{6}}\non\\[2mm]
   &\D\quad+ C(\bar{\rho})\|\rho^{\frac{1}{2}}\dot{u}\|_{L^{2}}A_{5}^{\frac{1}{3}}(\sigma(T))
   \|H\|_{L^{3}}\|\nabla^{2}H\|_{L^{2}}
   -\frac{1}{2\mu+\lambda}\intt Pu\cdot\nabla G+\frac{1}{2(2\mu+\lambda)}\intt \mdiv uP^{2}\non\\[2mm]
   &\D\quad-\frac{1}{2(2\mu+\lambda)}\intt Pu\cdot\nabla|H|^{2}+C(\bar{\rho})(\gamma-1)\|\nabla u\|_{L^{2}}^{2}
   +C(\bar{\rho})\|H\|_{L^{3}}\|H_{t}\|_{L^{2}}\|\rho^{\frac{1}{2}}\dot{u}\|_{L^{2}}
   \non\\[2mm]
   &\D\quad+C\|H\|_{L^{3}}\|H_{t}\|_{L^{2}}\|P\|_{L^{6}}+C\|H\|_{L^{3}}^{2}\|H_{t}\|_{L^{2}}\|\nabla^{2}H\|_{L^{2}}\non\\[2mm]
   &\D\leq \frac{1}{8}\|\rho^{\frac{1}{2}}\dot{u}\|_{L^{2}}^{2}
   +C(\bar{\rho})\|\rho^{\frac{1}{2}}\dot{u}\|_{L^{2}}^{2}A_{5}^{\frac{1}{3}}(\sigma(T))
   +C(\bar{\rho})A_{5}^{\frac{1}{3}}(\sigma(T))\|P\|_{L^{6}}^{2}\non\\[2mm]
   &\D\quad +C(\bar{\rho})A_{5}^{\frac{1}{3}}(\sigma(T))\|H\|_{L^{3}}^{2}\|\nabla^{2}H\|_{L^{2}}^{2}
   +C\|P\|_{L^{3}}\|\nabla u\|_{L^{2}}\left(\|\rho^{\frac{1}{2}}\dot{u}\|_{L^{2}}
   +\|H\|_{L^{3}}\|\nabla^{2}H\|_{L^{2}}\right)\non\\[2mm]
   &\D\quad +C\|P\|_{L^{4}}^{2}\|\nabla u\|_{L^{2}}^{2}+C\|P\|_{L^{4}}^{2}
   +C(\bar{\rho})\|\nabla u\|_{L^{2}}^{2}\|H\|_{L^{3}}^{2}+C(\bar{\rho})\|\nabla H\|_{L^{2}}^{2}
   +C(\bar{\rho})(\gamma-1)\|\nabla u\|_{L^{2}}^{2}\non\\[2mm]
   &\D\quad+C(\bar{\rho})\left((\gamma-1)^{\frac{1}{6}}E_{0}^{\frac{1}{2}}\right)^{\frac{2}{27}}
   \|H_{t}\|_{L^{2}}^{2}+C\|P\|_{L^{6}}^{2}
   +C\left((\gamma-1)^{\frac{1}{6}}E_{0}^{\frac{1}{2}}\right)^{\frac{2}{27}}
   \|\nabla^{2}H\|_{L^{2}}^{2}\non\\[2mm]
   &\D\leq \frac{1}{4}\|\rho^{\frac{1}{2}}\dot{u}\|_{L^{2}}^{2}
   +C_{3}(\bar{\rho})\|\rho^{\frac{1}{2}}\dot{u}\|_{L^{2}}^{2}A_{5}^{\frac{1}{3}}(\sigma(T))
   +C(\bar{\rho})A_{5}^{\frac{1}{3}}(\sigma(T))
   (\gamma-1)^{\frac{1}{3}}E_{0}^{\frac{1}{3}}\non\\[2mm]
   &\D\quad+C(\bar{\rho})A_{5}^{\frac{1}{3}}(\sigma(T))
  \left((\gamma-1)^{\frac{1}{6}}E_{0}^{\frac{1}{2}}\right)^{\frac{2}{27}}
   \|\nabla^{2}H\|_{L^{2}}^{2}
   +C(\gamma-1)^{\frac{2}{3}}E_{0}^{\frac{2}{3}}\|\nabla u\|_{L^{2}}^{2}\non\\
   &\D\quad+C(\gamma-1)^{\frac{2}{3}}E_{0}^{\frac{2}{3}}\left((\gamma-1)^{\frac{1}{6}}E_{0}^{\frac{1}{2}}\right)^{\frac{2}{27}}
   \|\nabla u\|_{L^{2}}^{2}
   +C(\gamma-1)^{\frac{2}{3}}E_{0}^{\frac{2}{3}}
   \left((\gamma-1)^{\frac{1}{6}}E_{0}^{\frac{1}{2}}\right)^{\frac{2}{27}}
   \|\nabla^{2}H\|_{L^{2}}^{2}\non\\[2mm]
   &\D\quad+C(\gamma-1)^{\frac{1}{2}}E_{0}^{\frac{1}{2}}\|\nabla u\|_{L^{2}}^{2}
   +(\gamma-1)^{\frac{1}{2}}E_{0}^{\frac{1}{2}}
   +C(\bar{\rho})\left((\gamma-1)^{\frac{1}{6}}E_{0}^{\frac{1}{2}}\right)^{\frac{2}{27}}
   \|\nabla u\|_{L^{2}}^{2}\non\\
   &\D\quad+C(\bar{\rho})\|\nabla H\|_{L^{2}}^{2}
   +C(\bar{\rho})(\gamma-1)\|\nabla u\|_{L^{2}}^{2}
   +(\gamma-1)^{\frac{1}{3}}E_{0}^{\frac{1}{3}}\non\\[2mm]
   &\D\quad+C\left((\gamma-1)^{\frac{1}{6}}E_{0}^{\frac{1}{2}}\right)^{\frac{2}{27}}
   \|H_{t}\|_{L^{2}}^{2}
   +C(\bar{\rho})\left((\gamma-1)^{\frac{1}{6}}E_{0}^{\frac{1}{2}}\right)^{\frac{2}{27}}
   \|\nabla^{2}H\|_{L^{2}}^{2}.
 \end{align}
 Integrating (\ref{2dbu-E3.64}) over $(0,\sigma(T))$, we get 
 \begin{align}
   & \D\quad\frac{\mu}{2}\sup_{0\leq t\leq\sigma(T)}\|\nabla u\|_{L^{2}}^{2}+\frac{1}{2}\int_{0}^{\sigma(T)}\intt\rho|\dot{u}|^{2}\non\\[2mm]
   &\D\leq C(\bar{\rho})A_{5}^{\frac{1}{3}}(\sigma(T))
   (\gamma-1)^{\frac{1}{3}}E_{0}^{\frac{1}{3}}
   +C(\bar{\rho})A_{5}^{\frac{1}{3}}(\sigma(T))
  \left((\gamma-1)^{\frac{1}{6}}E_{0}^{\frac{1}{2}}\right)^{\frac{2}{27}}
   \int_{0}^{\sigma(T)}\|\nabla^{2}H\|_{L^{2}}^{2}\non\\[2mm]
   &\D\quad+C(\gamma-1)^{\frac{2}{3}}E_{0}^{\frac{2}{3}}\int_{0}^{\sigma(T)}\|\nabla u\|_{L^{2}}^{2}
   +C(\gamma-1)^{\frac{2}{3}}E_{0}^{\frac{2}{3}}
   \left((\gamma-1)^{\frac{1}{6}}E_{0}^{\frac{1}{2}}\right)^{\frac{2}{27}}
   \int_{0}^{\sigma(T)}\|\nabla u\|_{L^{2}}^{2}\non\\[2mm]
   &\D\quad+C(\gamma-1)^{\frac{2}{3}}E_{0}^{\frac{2}{3}}
   \left((\gamma-1)^{\frac{1}{6}}E_{0}^{\frac{1}{2}}\right)^{\frac{2}{27}}
   \int_{0}^{\sigma(T)}\|\nabla^{2}H\|_{L^{2}}^{2}
   +(\gamma-1)^{\frac{1}{2}}E_{0}^{\frac{1}{2}}\int_{0}^{\sigma(T)}\|\nabla u\|_{L^{2}}^{2}\non\\[2mm]
   &\D\quad+(\gamma-1)^{\frac{1}{2}}E_{0}^{\frac{1}{2}}
   +C(\bar{\rho})\left((\gamma-1)^{\frac{1}{6}}E_{0}^{\frac{1}{2}}\right)^{\frac{2}{27}}
   \int_{0}^{\sigma(T)}\|\nabla u\|_{L^{2}}^{2}\non\\[2mm]
   &\D\quad+C(\bar{\rho})\int_{0}^{\sigma(T)}\|\nabla H\|_{L^{2}}^{2}
   +C(\bar{\rho})(\gamma-1)\int_{0}^{\sigma(T)}\|\nabla u\|_{L^{2}}^{2}
   +(\gamma-1)^{\frac{1}{3}}E_{0}^{\frac{1}{3}}\non\\[2mm]
   &\D\quad+C\left((\gamma-1)^{\frac{1}{6}}E_{0}^{\frac{1}{2}}\right)^{\frac{2}{27}}
   \int_{0}^{\sigma(T)}\|H_{t}\|_{L^{2}}^{2}
   +C(\bar{\rho})\left((\gamma-1)^{\frac{1}{6}}E_{0}^{\frac{1}{2}}\right)^{\frac{2}{27}}
   \int_{0}^{\sigma(T)}\|\nabla^{2}H\|_{L^{2}}^{2}\non\\[2mm]
   &\D\quad+\intt\Big(\mu|\nabla u|^{2}+(\mu+\lambda)|\mdiv u|^{2}-P\mdiv u-
   \Big(H\cdot\nabla H-\frac{1}{2}\nabla|H|^{2}\Big)\Big)\cdot u\Big|_{t=0}\non\\[2mm]
   &\D\leq C(\bar{\rho})(K_{2}M_{2}+K_{3}+K_{2}^{2}+M_{2}^{\frac{3}{4}}M_{1}^{\frac{1}{2}}+1)+C(\bar{\rho})(\gamma-1)^{\frac{2}{9}}+CK_{3}(\gamma-1)^{\frac{4}{9}}
   \non\\[2mm]
   &\D\quad+CK_{2}M_{2}(\gamma-1)^{\frac{4}{9}}+CK_{3}(\gamma-1)^{\frac{1}{3}}+(\gamma-1)^{\frac{1}{3}}+C(K_{3}+1)(\gamma-1),
 \end{align}
provided $\D (\gamma-1)^{\frac{1}{6}}E_{0}^{\frac{1}{2}}\leq \min\{(4C_{3}(\bar{\rho}))^{-27},1\}$, then we get (\ref{2dbu-E3.0061}).

  Taking $m=1$ in (\ref{2dbu-E3.042}), one has
 \begin{align}\label{2dbu-E3.0071}
   &\D \quad\frac{1}{2}\rf\intt\sigma\rho|\dot{u}|^{2}
   +\intt\sigma|\nabla\dot{u}|^{2}
   -C_{1}\sigma\left((\gamma-1)^{\frac{1}{6}}E_{0}^{\frac{1}{2}}\right)^{\frac{2}{27}}
   \|\nabla H_{t}\|_{L^{2}}^{2}\non\\
   &\D\leq C\sigma'\intt\rho|\dot{u}|^{2}+C\mu\intt\sigma|\nabla u|^{4}
   +C\gamma^{2}\intt\sigma|P\nabla u|^{2}\non\\[2mm]
  &\D\quad+C\sigma\left(\|\nabla u\|_{L^{2}}^{4}
  +\|\nabla H\|_{L^{2}}^{4}\right)\|\nabla^{2}H\|_{L^{2}}^{2}.
 \end{align}
 Integrating (\ref{2dbu-E3.0071}) over $(0,\sigma(T))$, we get
 \begin{align}\label{2dbu-E0.069}
   &\D \quad\intt\sigma\rho|\dot{u}|^{2}+\int_{0}^{\sigma(T)}\intt\sigma|\nabla\dot{u}|^{2}
   -C_{1}\int_{0}^{\sigma(T)}\sigma\left((\gamma-1)^{\frac{1}{6}}E_{0}^{\frac{1}{2}}\right)^{\frac{2}{27}}
   \|\nabla H_{t}\|_{L^{2}}^{2}\non\\
   &\D\leq C\int_{0}^{\sigma(T)}\intt\rho|\dot{u}|^{2}+C\mu\int_{0}^{\sigma(T)}
   \intt\sigma|\nabla u|^{4}
   +C\gamma^{2}\int_{0}^{\sigma(T)}\intt\sigma|P\nabla u|^{2}\non\\[2mm]
  &\D\quad+C\int_{0}^{\sigma(T)}\sigma\left(\|\nabla u\|_{L^{2}}^{4}
  +\|\nabla H\|_{L^{2}}^{4}\right)\|\nabla^{2}H\|_{L^{2}}^{2}\non\\
  &\D\leq CK_{6}+CK_{6}^{\frac{3}{2}}+CK_{6}^{\frac{1}{2}}K_{2}M_{2}+CK_{6}^{\frac{1}{2}}(\gamma-1)^{\frac{1}{3}}
  +C\left(\int_{0}^{\sigma(T)}\intt P^{4}\right)^{\frac{1}{2}}
  \left(\int_{0}^{\sigma(T)}\intt \sigma^{2}|\nabla u|^{4}\right)^{\frac{1}{2}}\non\\
  &\D\quad+C\left[\Big(\sup_{0\leq t\leq \sigma(T)}\|\nabla u\|_{L^{2}}^{2}\Big)^{2}
  +\Big(\sup_{0\leq t\leq \sigma(T)}\|\nabla H\|_{L^{2}}^{2}\Big)^{2}\right]
  \int_{0}^{\sigma(T)}\sigma\|\nabla^{2}H\|_{L^{2}}^{2}\non\\[2mm]
  &\D\leq  CK_{6}+CK_{6}^{\frac{3}{2}}+CK_{6}^{\frac{1}{2}}K_{2}M_{2}+CK_{6}^{\frac{1}{2}}(\gamma-1)^{\frac{1}{3}}
  +C(K_{6}^{2}+K_{2}^{2}M_{2}^{2})\non\\
  &\D\quad+C(\gamma-1)^{\frac{1}{3}}(K_{6}^{\frac{3}{2}}+K_{6}^{\frac{1}{2}}K_{2}M_{2}+K_{6}^{\frac{1}{2}}(\gamma-1)^{\frac{1}{3}})^{\frac{1}{2}}\non\\
  &\D \triangleq K_{7}^{1},
 \end{align}
 where we have used the condition $\D (\gamma-1)^{\frac{1}{6}}E_{0}^{\frac{1}{2}}\leq 1$ and the following estimate
 \begin{align}
  &\D\quad \int_{0}^{\sigma(T)}\intt\sigma|\nabla u|^{4}\non\\[2mm]
  &\D\leq\sup_{0\leq t\leq \sigma(T)}\|\nabla u\|_{L^{2}}
  \int_{0}^{\sigma(T)}\sigma\|\nabla u\|_{L^{6}}^{3}\non\\[2mm]
  &\D\leq CK_{6}^{\frac{1}{2}}\int_{0}^{\sigma(T)}\sigma\Big(\|\rho^{\frac{1}{2}}\dot{u}\|_{L^{2}}^{3}
  +\|P\|_{L^{6}}^{3}+\|H\cdot\nabla H\|_{L^{2}}^{3}\Big)\non\\[2mm]
  &\D\leq CK_{6}^{\frac{1}{2}}\Big(\sup_{0\leq t\leq \sigma(T)}\sigma^{2}\|\rho^{\frac{1}{2}}\dot{u}
  \|_{L^{2}}^{2}\Big)^{\frac{1}{2}}
  \int_{0}^{\sigma(T)}\|\rho^{\frac{1}{2}}\dot{u}\|_{L^{2}}^{2}
  +CK_{6}^{\frac{1}{2}}(\gamma-1)^{\frac{1}{2}}E_{0}^{\frac{1}{2}}\non\\[2mm]
  &\D\quad+CK_{6}^{\frac{1}{2}}\int_{0}^{\sigma(T)}\sigma\|H\|_{L^{3}}^{3}
  \|\nabla^{2}H\|_{L^{2}}^{3}\non\\[2mm]
  &\D\leq  CK_{6}^{\frac{1}{2}}(A_{2}(T))^{\frac{1}{2}}K_{6}
  +CK_{6}^{\frac{1}{2}}(\gamma-1)^{\frac{1}{2}}E_{0}^{\frac{1}{2}}\non\\[2mm]
  &\D\quad +CK_{6}^{\frac{1}{2}}
  \Big(\sup_{0\leq t\leq \sigma(T)}\sigma^{2}\|\nabla^{2}H\|_{L^{2}}^{2}\Big)^{\frac{1}{2}}
  \int_{0}^{\sigma(T)}\|\nabla^{2}H\|_{L^{2}}^{2}\non\\[2mm]
  &\D\leq CK_{6}^{\frac{1}{2}}(A_{2}(T))^{\frac{1}{2}}K_{6}
  +CK_{6}^{\frac{1}{2}}(\gamma-1)^{\frac{1}{2}}E_{0}^{\frac{1}{2}}
  +CK_{6}^{\frac{1}{2}}(A_{2}(T))^{\frac{1}{2}}K_{2}M_{2}\non\\[2mm]
  &\D\leq CK_{6}^{\frac{3}{2}}+CK_{6}^{\frac{1}{2}}K_{2}M_{2}+CK_{6}^{\frac{1}{2}}(\gamma-1)^{\frac{1}{3}}.
\end{align}
 Similarly,  Taking $m=1$ in (\ref{2dbu-E3.048}), we get
 \begin{align}
   &\D\quad\rf\intt\sigma|H_{t}|^{2}+\intt\sigma|\nabla H_{t}|^{2}
   -C_{2}\sigma\left((\gamma-1)^{\frac{1}{6}}E_{0}^{\frac{1}{2}}\right)^{\frac{1}{27}}
   (\|\nabla \dot{u}\|_{L^{2}}^{2}+\|\nabla H_{t}\|_{L^{2}}^{2})\non\\[2mm]
   &\D\leq\frac{C}{2}\sigma'\intt|H_{t}|^{2}
   +C\sigma\|\nabla u\|_{L^{2}}^{2}\|\rho^{\frac{1}{2}}\dot{u}\|_{L^{2}}^{2}\|\nabla H\|_{L^{2}}^{2}\non\\[2mm]
   &\D\quad+C\sigma\left(\|\nabla u\|_{L^{2}}^{4}
   +\|\nabla H\|_{L^{2}}^{4}\right)\left(\|H\|_{L^{3}}^{2}\|\nabla^{2}H\|_{L^{2}}^{2}
   +\|H_{t}\|_{L^{2}}^{2}\right)\non\\[2mm]
   &\D\quad+C\sigma\|\nabla u\|_{L^{2}}^{2}\|\nabla H\|_{L^{2}}^{2}
   (\gamma-1)^{\frac{1}{3}}E_{0}^{\frac{1}{3}}
 \end{align}
 and
 \begin{align}\label{2dbu-E0.071}
   &\D\quad\intt\sigma|H_{t}|^{2}+\int_{0}^{\sigma(T)}\intt\sigma|\nabla H_{t}|^{2}
   -C_{2}\int_{0}^{\sigma(T)}\intt\sigma\left((\gamma-1)^{\frac{1}{6}}E_{0}^{\frac{1}{2}}\right)^{\frac{1}{27}}(\|\nabla \dot{u}\|_{L^{2}}^{2}
   +\|\nabla H_{t}\|_{L^{2}}^{2})\non\\[2mm]
   &\D\leq C\int_{0}^{\sigma(T)}\intt|H_{t}|^{2}
   +C\int_{0}^{\sigma(T)}\sigma\|\nabla u\|_{L^{2}}^{2}\|\rho^{\frac{1}{2}}\dot{u}\|_{L^{2}}^{2}\|\nabla H\|_{L^{2}}^{2}\non\\[2mm]
   &\D\quad+C\int_{0}^{\sigma(T)}\sigma\left(\|\nabla u\|_{L^{2}}^{4}
   +\|\nabla H\|_{L^{2}}^{4}\right)\left(\|H\|_{L^{3}}^{2}\|\nabla^{2}H\|_{L^{2}}^{2}
   +\|H_{t}\|_{L^{2}}^{2}\right)\non\\[2mm]
   &\D\quad+C\int_{0}^{\sigma(T)}\sigma\|\nabla u\|_{L^{2}}^{2}\|\nabla H\|_{L^{2}}^{2}
   (\gamma-1)^{\frac{1}{3}}E_{0}^{\frac{1}{3}}\non\\[2mm]
   &\D \leq CK_{2}M_{2}+C\Big(\sup_{0\leq t\leq \sigma(T)}\|\nabla u\|_{L^{2}}^{2}\Big)
   \Big(\sup_{0\leq t\leq \sigma(T)}\sigma\|\nabla H\|_{L^{2}}^{2}\Big)
   \int_{0}^{\sigma(T)}\|\rho^{\frac{1}{2}}\dot{u}\|_{L^{2}}^{2}\non\\[2mm]
   &\D \quad+C\left((\gamma-1)^{\frac{1}{6}}E_{0}^{\frac{1}{2}}\right)^{\frac{2}{27}}
   \Big(\sup_{0\leq t\leq \sigma(T)}\|\nabla u\|_{L^{2}}^{2}\Big)^{2}
   \int_{0}^{\sigma(T)}\sigma\|\nabla^{2}H\|_{L^{2}}^{2}\non\\[2mm]
   &\D\quad+C\Big(\sup_{0\leq t\leq \sigma(T)}\|\nabla H\|_{L^{2}}^{2}\Big)
   \Big(\sup_{0\leq t\leq \sigma(T)}\sigma\|\nabla H\|_{L^{2}}^{2}\Big)
   \int_{0}^{\sigma(T)}\|\nabla^{2} H\|_{L^{2}}^{2}\non\\[2mm]
   &\D\quad+C\bi(\sup_{0\leq t\leq \sigma(T)}\|\nabla u\|_{L^{2}}^{2}
   +\sup_{0\leq t\leq \sigma(T)}\|\nabla H\|_{L^{2}}^{2}\bi)^{2}
   \int_{0}^{\sigma(T)}\sigma\|H_{t}\|_{L^{2}}^{2}\non\\[2mm]
   &\D\quad+C\bi(\sup_{0\leq t\leq \sigma(T)}\sigma\|\nabla H\|_{L^{2}}^{2}\bi)
   \bi(\sup_{0\leq t\leq \sigma(T)}\|\nabla u\|_{L^{2}}^{2}\bi)
   (\gamma-1)^{\frac{1}{3}}E_{0}^{\frac{1}{3}}\non\\[2mm]
   &\D\leq CK_{2}M_{2}+CK_{6}^{2}K_{2}^{2}+CK_{2}^{2}M_{2}^{2}+C(K_{6}^{2}+CK_{2}^{2}M_{2}^{2})K_{2}^{2}+CK_{6}K_{2}^{2}(\gamma-1)^{\frac{2}{9}}\non\\[2mm]
   &\D\triangleq K_{7}^{2}.
 \end{align}
 Combining (\ref{2dbu-E0.069}) and (\ref{2dbu-E0.071}), we deduce
  \begin{align}
  &\D \quad\sup_{0\leq t\leq \sigma(T)} t\bi(\|\rho^{\frac{1}{2}}\dot{u}\|_{L^{2}}^{2}+\|H_{t}\|_{L^{2}}^{2}\bi)
  +\int_{0}^{\sigma(T)} t\bi(\|\nabla\dot{u}\|_{L^{2}}^{2}+\|\nabla H_{t}\|_{L^{2}}^{2}\bi)\leq K_{7},
 \end{align}
 provided that
 \begin{align}
   \D (\gamma-1)^{\frac{1}{6}}E_{0}^{\frac{1}{2}}\leq \min\left\{\varepsilon_{2},(4C_{1})^{-\frac{27}{2}}+(4C_{2})^{-27},1\right\},
 \end{align}
where $\D K_{7}=2\max\{K_{7}^{1},K_{7}^{2}\}$. And this leads to (\ref{2dbu-E3.0062}). By (\ref{2dbu-E3.02}),
 (\ref{2dbu-E3.0061}) and (\ref{2dbu-E3.0062}), we can get (\ref{2dbu-E3.0063}) and (\ref{2dbu-E3.0064}). Then
 we finish the proof of Lemma \ref{2dbu-L3.7}.
 \end{proof}
 \begin{lem}\label{2dbu-L3.8}
Under the same assumption as in Proposition \ref{2dbu-P3.1}, we get that
\begin{align}
  \D A_{4}(\sigma(T))+ A_{5}(\sigma(T))\leq \left((\gamma-1)^{\frac{1}{6}}E_{0}^{\frac{1}{2}}\right)^{\frac{1}{9}},
\end{align}
provided that
\begin{align}
  \D (\gamma-1)^{\frac{1}{6}}E_{0}^{\frac{1}{2}}\leq \varepsilon_{4}\triangleq\min\left\{\varepsilon_{3},\frac{1}{K_{6}^{54}},
  \left(C(\bar{\rho})(\gamma-1)^{\frac{1}{3}}K_{6}^{\frac{3}{2}}\right)^{-\frac{9}{8}},1\right\}.
\end{align}
\end{lem}
\begin{proof}
It follows from (\ref{2dbu-E3.02}) and (\ref{2dbu-E3.0062}) that
\begin{align}
 \D A_{4}(\sigma(T))&\triangleq \sup_{0\leq t\leq \sigma(T)}\sigma^{\frac{1}{3}}\|\nabla u\|_{L^{2}}^{2}\non\\
 &\leq \bi(\sup_{0\leq t\leq \sigma(T)}\sigma\|\nabla u\|_{L^{2}}^{2}\bi)^{\frac{1}{4}}\bi(\sup_{0\leq t\leq \sigma(T)}\|\nabla u\|_{L^{2}}^{2}\bi)^{\frac{3}{4}}\non\\
 &\leq K_{6}^{\frac{3}{4}}\left((\gamma-1)^{\frac{1}{6}}E_{0}^{\frac{1}{2}}\right)^{\frac{1}{8}}\non\\
 &\leq \left((\gamma-1)^{\frac{1}{6}}E_{0}^{\frac{1}{2}}\right)^{\frac{1}{9}},
\end{align}
provided $(\gamma-1)^{\frac{1}{6}}E_{0}^{\frac{1}{2}}\leq \min\left\{\varepsilon_{3},\frac{1}{K_{6}^{54}},1\right\}$.

 Now, to end up the proof of Lemma \ref{2dbu-L3.8}, it remains to estimate $A_{5}(\sigma(T))$. Due to (\ref{2dbu-E3.0062}), we deduce that
 \begin{align}
   \D A_{5}(\sigma(T))&=\sup_{0\leq t\leq \sigma(T)}\intt \rho|u|^{3}\non\\[2mm]
   \D& \leq\sup_{0\leq t\leq \sigma(T)}\left(\|\rho\|_{L^{2}}\|u\|_{L^{6}}^{3}\right)\non\\[2mm]
   \D&\leq C(\bar{\rho})(\gamma-1)^{\frac{1}{2}}E_{0}^{\frac{1}{2}}K_{6}^{\frac{3}{2}}\non\\[2mm]
   \D&\leq C(\bar{\rho})K_{6}^{\frac{3}{2}}(\gamma-1)^{\frac{1}{3}}(\gamma-1)^{\frac{1}{6}}E_{0}^{\frac{1}{2}}\non\\[2mm]
   &\D\leq \left((\gamma-1)^{\frac{1}{6}}E_{0}^{\frac{1}{2}}\right)^{\frac{1}{9}},
 \end{align}
 provided that
 \begin{align}
  \D (\gamma-1)^{\frac{1}{6}}E_{0}^{\frac{1}{2}}\leq \min\left\{\left(C(\bar{\rho})(\gamma-1)^{\frac{1}{3}}K_{6}^{\frac{3}{2}}\right)^{-\frac{9}{8}},1\right\}.
 \end{align}
\end{proof}
\begin{lem}\label{2dbu-L3.9}
Under the same assumption as in Proposition \ref{2dbu-P3.1}, we get that
\begin{align}\label{2dbu-E3.85}
  \D A_{1}(T)+ A_{2}(T)\leq \left((\gamma-1)^{\frac{1}{6}}E_{0}^{\frac{1}{2}}\right)^{\frac{1}{2}},
\end{align}
provided
\begin{align}
  \D(\gamma-1)^{\frac{1}{6}}E_{0}^{\frac{1}{2}}\leq \varepsilon_{5}\triangleq\min\left\{\varepsilon_{4},(C_{4}K_{9})^{-1},1\right\},
\end{align}
where
\begin{align}
   \D K_{9}&=(K_{4}+K_{5})+K_{3}+ K_{8}^{\frac{1}{2}}(\gamma-1)^{\frac{1}{12\times24}}+\sqrt[4]{K_{8}^{\tilde{2}}K_{8}}
    (\gamma-1)^{\frac{1}{12\times24}}\non\\
    \D&\quad+\sqrt{K_{8}^{\tilde{2}}}K_{8}^{\frac{1}{2}}+K_{8},\non\\
   \D K_{8}^{\tilde{1}}&=C(K_{6}^{\frac{1}{2}}+1)(K_{7}+1)(\gamma-1)^{\frac{4}{9}}
   +C(K_{7}+1)(\gamma-1)^{\frac{7}{9}}+CK_{2}^{\frac{9}{4}}M_{2}^{\frac{3}{4}}(K_{6}^{\frac{1}{2}}+1)\non\\
   \D&\quad +CK_{2}^{\frac{9}{4}}M_{2}^{\frac{3}{4}}(\gamma-1)^{\frac{1}{3}}
   +CK_{2}^{\frac{5}{4}}M_{2}^{\frac{3}{4}}(K_{7}+1)(\gamma-1)^{\frac{4}{9}}
    +CK_{2}^{\frac{7}{2}}M_{2}^{\frac{3}{2}},\non\\
    \D K_{8}^{\tilde{2}}&=C(\bar{\rho})(\gamma-1)^{\frac{2}{3}}+CK_{2}^{\frac{15}{4}}M_{2}^{\frac{3}{4}}+K_{8}^{\tilde{1}}.
 \end{align}
\end{lem}
\begin{proof}
 From (\ref{2dbu-E3.030}) and (\ref{2dbu-E3.33}), we have
\begin{align}
   \D& \quad A_{1}(T)+A_{2}(T)\non\\[2mm]
   \D& \leq C(K_{4}+K_{5})\left((\gamma-1)^{\frac{1}{6}}E_{0}^{\frac{1}{2}}\right)^{\frac{11}{18}}
   +C\int_{0}^{\sigma(T)}\intt|\nabla u|^{2}+C\int_{0}^{T}\sigma\intt|\nabla u|^{3}\non\\[2mm]
  \D&\quad+C\gamma\int_{0}^{T}\sigma\intt P|\nabla u|^{2}
  +C\gamma^{2}\int_{0}^{T}\intt\sigma^{2}|P\nabla u|^{2}+C\int_{0}^{T}\intt\sigma^{2}|\nabla u|^{4}\non\\
 \D & \leq C(K_{4}+K_{5})\left((\gamma-1)^{\frac{1}{6}}E_{0}^{\frac{1}{2}}\right)^{\frac{11}{18}}
  +\sum_{i=1}^{5}\mathscr{T}_{i}.
 \end{align}
 An application of (\ref{2dbu-E2.07}) gives
 \begin{align}\label{2dbu-E3.088}
   \D\mathscr{T}_{5}&=C\int_{0}^{T}\intt\sigma^{2}|\nabla u|^{4}\non\\[2mm]
   &\D\leq C\int_{0}^{T}\sigma^{2}\|\nabla u\|_{L^{2}}\|\rho\dot{u}\|_{L^{2}}^{3}
  +C\int_{0}^{T}\sigma^{2}\|P\|_{L^{2}}\|\rho\dot{u}\|_{L^{2}}^{3}
  +C\int_{0}^{T}\sigma^{2}\|H\|_{L^{2}}^{\frac{1}{2}}
  \|\nabla H\|_{L^{2}}^{\frac{3}{2}}\|\rho\dot{u}\|_{L^{2}}^{3}
    \non\\[2mm]
  \D&\quad
  +C\int_{0}^{T}\sigma^{2}\|\nabla u\|_{L^{2}}\|H\|_{L^{3}}^{3}\|\nabla^{2} H\|_{L^{2}}^{3}
  +C\int_{0}^{T}\sigma^{2}\|P\|_{L^{2}}\|H\|_{L^{3}}^{3}\|\nabla^{2} H\|_{L^{2}}^{3}\non\\[2mm]
  \D&\quad
  +C\int_{0}^{T}\sigma^{2}\|H\|_{L^{2}}^{\frac{1}{2}}
  \|\nabla H\|_{L^{2}}^{\frac{3}{2}}\|H\|_{L^{3}}^{3}\|\nabla^{2} H\|_{L^{2}}^{3}
  +C\int_{0}^{T}\sigma^{2}\|P\|_{L^{4}}^{4}\non\\[2mm]
    \D&=\sum_{i=1}^{7}J_{i}.
 \end{align}
 Thanks to (\ref{2dbu-E1.011}), (\ref{2dbu-E3.02}), (\ref{2dbu-E3.03}), (\ref{2dbu-E3.0063}) and (\ref{2dbu-E3.0064}), we have
 \begin{align}
  \D J_{1}&\leq C(\bar{\rho}) \bi(\sup_{0 \leq t\leq T}\|\nabla u\|_{L^{2}}^{2}\bi)^{\frac{1}{2}}
  \bi(\sup_{0 \leq t\leq T}\sigma^{2}\|\rho^{\frac{1}{2}}\dot{u}\|_{L^{2}}^{2}\bi)^{\frac{1}{2}}
  \int_{0}^{T}\sigma\|\rho\|_{L^{^{3}}}^{2}\|\nabla\dot{u}\|_{L^{2}}^{2}\non\\[2mm]
  &\D\leq C(\bar{\rho})(K_{6}^{\frac{1}{2}}+1)\left((\gamma-1)^{\frac{1}{6}}
  E_{0}^{\frac{1}{2}}\right)^{\frac{1}{4}}(\gamma-1)^{\frac{2}{3}}E_{0}^{\frac{2}{3}}
  (K_{7}+1)\non\\
  &\D\leq C(\bar{\rho})(K_{6}^{\frac{1}{2}}+1)(K_{7}+1)(\gamma-1)^{\frac{17}{24}}E_{0}^{\frac{19}{24}}\non\\
  &\D\leq C(\bar{\rho})(K_{6}^{\frac{1}{2}}+1)(K_{7}+1)(\gamma-1)^{\frac{4}{9}}
  \left((\gamma-1)^{\frac{1}{6}}E_{0}^{\frac{1}{2}}\right)^{\frac{3}{2}}.
 \end{align}
 Similarly, we deduce
 \begin{align}
     \D J_{2}&\leq C\int_{0}^{T}\sigma^{2}\|P\|_{L^{2}}\|\rho\dot{u}\|_{L^{2}}^{3}\non\\[2mm]
   &\leq C\bi(\sup_{0 \leq t\leq T}\sigma^{2}\|\rho^{\frac{1}{2}}\dot{u}\|_{L^{2}}^{2}\bi)^{\frac{1}{2}}
   \bi(\sup_{0 \leq t\leq T}\|P\|_{L^{2}}\Big)
   \int_{0}^{T}\sigma\|\rho\|_{L^{^{3}}}^{2}\|\nabla\dot{u}\|_{L^{2}}^{2}\non\\[2mm]
   &\leq C(\bar{\rho})(K_{7}+1)(\gamma-1)^{\frac{29}{24}}E_{0}^{\frac{31}{24}}\non\\
   &\leq C(\bar{\rho})(K_{7}+1)(\gamma-1)^{\frac{7}{9}}
   \left((\gamma-1)^{\frac{1}{6}}E_{0}^{\frac{1}{2}}\right)^{\frac{3}{2}}.
  \end{align}
 It follows from (\ref{2dbu-E1.011}), (\ref{2dbu-E3.02}), (\ref{2dbu-E3.08}), (\ref{2dbu-E3.9}) and (\ref{2dbu-E3.0063})-(\ref{2dbu-E3.0064}) that
 \begin{align}
   \D J_{3}&\leq C\int_{0}^{T}\sigma^{2}\|H\|_{L^{2}}^{\frac{1}{2}}
   \|\nabla H\|_{L^{2}}^{\frac{3}{2}}\|\rho\dot{u}\|_{L^{2}}^{3}\non\\[2mm]
   &\leq CK_{2}^{\frac{1}{2}}\bi(\sup_{0 \leq t\leq T}\sigma^{2}\|\rho^{\frac{1}{2}}\dot{u}\|_{L^{2}}^{2}\bi)^{\frac{1}{2}}
   \bi(\sup_{0 \leq t\leq T}\|H\|_{L^{2}}^{2}\Big)^{\frac{1}{4}}
   \int_{0}^{T}\sigma\|\rho\dot{u}\|_{L^{2}}^{2}\non\\[2mm]
   &\leq  CK_{2}^{\frac{1}{2}}\left((\gamma-1)^{\frac{1}{6}}E_{0}^{\frac{1}{2}}\right)^{\frac{1}{2}}
   \int_{0}^{T}\sigma\|\rho\|_{L^{^{3}}}^{2}\|\nabla\dot{u}\|_{L^{2}}^{2}\non\\
   &\leq CK_{2}^{\frac{1}{2}}\left((\gamma-1)^{\frac{1}{6}}E_{0}^{\frac{1}{2}}\right)^{\frac{1}{2}}(\gamma-1)^{\frac{2}{3}}E_{0}^{\frac{2}{3}}\non\\
   &\leq CK_{2}^{\frac{1}{2}}(K_{7}+1)(\gamma-1)^{\frac{3}{4}}E_{0}^{\frac{11}{12}}\non\\
   &\leq CK_{2}^{\frac{1}{2}}(K_{7}+1)(\gamma-1)^{\frac{4}{9}}
   \left((\gamma-1)^{\frac{1}{6}}E_{0}^{\frac{1}{2}}\right)^{\frac{3}{2}},
 \end{align}
Moreover, $J_{4}$ to $J_{6}$ can be estimated, in a similar way, as follows:
 \begin{align}
   \D J_{4}&\leq C\int_{0}^{T}\sigma^{2}\|\nabla u\|_{L^{2}}\|H\|_{L^{3}}^{3}\|\nabla^{2} H\|_{L^{2}}^{3}\non\\[2mm]
   &\leq C\bi(\sup_{0 \leq t\leq T}\sigma^{2}\|\nabla^{2} H\|_{L^{2}}^{2}\bi)^{\frac{1}{2}}
   \bi(\sup_{0 \leq t\leq T}\|H\|_{L^{2}}^{\frac{3}{2}}\|\nabla H\|_{L^{2}}^{\frac{3}{2}}\Big)\bi(\sup_{0 \leq t\leq T}\|\nabla u\|_{L^{2}}^{2}\Big)^{\frac{1}{2}}\int_{0}^{T}\sigma\|\nabla^{2} H\|_{L^{2}}^{2}\non\\[2mm]
   &\leq CK_{2}^{\frac{9}{4}}M_{2}^{\frac{3}{4}}(K_{6}^{\frac{1}{2}}+1)\left((\gamma-1)^{\frac{1}{6}}E_{0}^{\frac{1}{2}}\right)^{\frac{3}{2}},
  \end{align}
  \begin{align}
   \D J_{5}&\leq C\int_{0}^{T}\sigma^{2}\|P\|_{L^{2}}\|H\|_{L^{3}}^{3}\|\nabla^{2} H\|_{L^{2}}^{3}\non\\[2mm]
   &\leq C\bi(\sup_{0 \leq t\leq T}\sigma^{2}\|\nabla^{2} H\|_{L^{2}}^{2}\bi)^{\frac{1}{2}}
   \bi(\sup_{0 \leq t\leq T}\|H\|_{L^{2}}^{\frac{3}{2}}\|\nabla H\|_{L^{2}}^{\frac{3}{2}}\Big)\bi(\sup_{0 \leq t\leq T}\|P\|_{L^{2}}\Big)
   \int_{0}^{T}\sigma\|\nabla^{2} H\|_{L^{2}}^{2}\non\\[2mm]
   &\leq CK_{2}^{\frac{9}{4}}M_{2}^{\frac{3}{4}}(\gamma-1)^{\frac{1}{2}}E_{0}^{\frac{1}{2}}
   \left((\gamma-1)^{\frac{1}{6}}E_{0}^{\frac{1}{2}}\right)^{\frac{3}{2}}\non\\
    &\leq CK_{2}^{\frac{9}{4}}M_{2}^{\frac{3}{4}}(\gamma-1)^{\frac{1}{3}}
    \left((\gamma-1)^{\frac{1}{6}}E_{0}^{\frac{1}{2}}\right)^{\frac{3}{2}},
 \end{align}
and
 \begin{align}\label{2dbu-E3.094}
   \D J_{6}&\leq C\int_{0}^{T}\sigma^{2}\|H\|_{L^{2}}^{\frac{1}{2}}
  \|\nabla H\|_{L^{2}}^{\frac{3}{2}}\|H\|_{L^{3}}^{3}\|\nabla^{2} H\|_{L^{2}}^{3}\non\\[2mm]
   &\leq C\bi(\sup_{0 \leq t\leq T}\sigma^{2}\|\nabla^{2} H\|_{L^{2}}^{2}\bi)^{\frac{1}{2}}
   \bi(\sup_{0 \leq t\leq T}\|H\|_{L^{3}}^{3}\Big)\bi(\sup_{0 \leq t\leq T}\|H\|_{L^{2}}^{\frac{1}{2}}
  \|\nabla H\|_{L^{2}}^{\frac{3}{2}}\Big)
   \int_{0}^{T}\sigma\|\nabla^{2} H\|_{L^{2}}^{2}\non\\[2mm]
   &\leq CK_{2}^{3}M_{2}^{\frac{3}{4}}\left((\gamma-1)^{\frac{1}{6}}E_{0}^{\frac{1}{2}}\right)^{\frac{3}{2}}.
 \end{align}
In order to obtain desired estimate on $\mathscr{T}_{5}$, it suffices to estimate $J_{7}$. One can deduce from $(\ref{2dbu-E1.1})_{1}$ that
\begin{align}\label{2dbu-E3.095}
   \D P_{t}+u\cdot\nabla P+\gamma P\mdiv u=0.
 \end{align}
 In terms of the effective viscous flux $G$, we can rewrite (\ref{2dbu-E3.095}) as
 \begin{align}\label{2dbu-E3.096}
   \D P_{t}+u\cdot\nabla P+\frac{\gamma}{2\mu+\lambda} PG+\frac{\gamma}{2(2\mu+\lambda)}P|H|^{2}+\frac{\gamma}{2\mu+\lambda}P^{2}=0.
 \end{align}
 Multiplying (\ref{2dbu-E3.096}) by $3\sigma^{2}P^{2}$ and integrating the resulting equality over $\mathbb{R}^{3}\times [0,T]$, we obtain
 \begin{align}\label{2dbu-E3.097}
  \D&\quad C\int_{0}^{T}\sigma^{2}\|P\|_{L^{4}}^{4}\non\\
  \D&\leq C\sup_{0 \leq t\leq T}(\sigma^{2}\|P\|_{L^{3}}^{3})
  +C\int_{0}^{T}\sigma\sigma'\|P\|_{L^{3}}^{3}
  +C\int_{0}^{T}\sigma^{2}\intt P^{3}G\non\\
  \D&\quad+C\int_{0}^{T}\sigma^{2}\intt P^{3}|H|^{2}\non\\
  \D&\leq C(\bar{\rho})(\gamma-1)E_{0}+\epsilon\int_{0}^{T}\sigma^{2}\|P\|_{L^{4}}^{4}+C\int_{0}^{T}\sigma^{2}\|G\|_{L^{4}}^{4}\non\\
  \D&\quad+\int_{0}^{T}\sigma^{2}\|H\|_{L^{8}}^{8}.
 \end{align}
 To handle the terms on the right hand side of (\ref{2dbu-E3.097}), we have by (\ref{2dbu-E2.02}), (\ref{2dbu-E3.08})-(\ref{2dbu-E3.9})
 and (\ref{2dbu-E3.019}) that
 \begin{align}\label{2dbu-E3.098}
 \D&\quad \int_{0}^{T}\sigma^{2}\|H\|_{L^{8}}^{8}\non\\
  \D&\leq C\int_{0}^{T}\sigma^{2}\|H\|_{L^{\infty}}\|H\|_{L^{6}}^{4}\|H\|_{L^{9}}^{3}\non\\
  \D&\leq C\int_{0}^{T}\sigma^{2}\|\nabla H\|_{L^{2}}^{\frac{9}{2}}\|\nabla^{2}H\|_{L^{2}}^{\frac{1}{2}}\|H\|_{L^{9}}^{3}\non\\
  \D&\leq C\bi(\sup_{0 \leq t\leq T}\sigma^{2}\|\nabla^{2} H\|_{L^{2}}^{2}\bi)^{\frac{1}{4}}
  \bi(\sup_{0 \leq t\leq T}\sigma\|\nabla H\|_{L^{2}}^{2}\bi)^{\frac{3}{2}}
  \bi(\sup_{0 \leq t\leq T}\|\nabla H\|_{L^{2}}^{2}\bi)^{\frac{3}{4}}\int_{0}^{T}\|H\|_{L^{9}}^{3}\non\\
  \D&\leq CK_{2}^{\frac{15}{4}}M_{2}^{\frac{3}{4}}\left((\gamma-1)^{\frac{1}{6}}E_{0}^{\frac{1}{2}}\right)^{\frac{13}{8}+\frac{1}{9}}.
  \end{align}
  Furthermore, we have also
  \begin{align}\label{2dbu-E3.099}
  \D&\quad C\int_{0}^{T}\sigma^{2}\|G\|_{L^{4}}^{4} \non\\[2mm]
  \D&\leq C\int_{0}^{T}\sigma^{2}\|G\|_{L^{2}}\|G\|_{L^{6}}^{3}\non\\[2mm]
  \D&\leq C\int_{0}^{T}\sigma^{2}\left((2\mu+\lambda)\|\nabla u\|_{L^{2}}+\|P\|_{L^{2}}\right)(\|\rho\dot{u}\|_{L^{2}}^{3}+\|H\cdot\nabla H\|_{L^{2}}^{3})
 \non\\[2mm]
  \D&\quad+C\int_{0}^{T}\sigma^{2}\|H\|_{L^{2}}^{\frac{1}{2}}\|\nabla H\|_{L^{2}}^{\frac{3}{2}}(\|\rho\dot{u}\|_{L^{2}}^{3}+\|H\cdot\nabla H\|_{L^{2}}^{3})\non\\[2mm]
  \D&\leq C\int_{0}^{T}\sigma^{2}\|\nabla u\|_{L^{2}}\|\rho\dot{u}\|_{L^{2}}^{3}+C\int_{0}^{T}\sigma^{2}\|P\|_{L^{2}}\|\rho\dot{u}\|_{L^{2}}^{3}\non\\[2mm]
  \D&\quad+C\int_{0}^{T}\sigma^{2}\|\nabla u\|_{L^{2}}\|H\|_{L^{3}}^{3}\|\nabla^{2}H\|_{L^{2}}^{3}
  +C\int_{0}^{T}\sigma^{2}\|P\|_{L^{2}}\|H\|_{L^{3}}^{3}\|\nabla^{2}H\|_{L^{2}}^{3}\non\\[2mm]
  \D&\quad+C\int_{0}^{T}\sigma^{2}\|H\|_{L^{2}}^{\frac{1}{2}}\|\nabla H\|_{L^{2}}^{\frac{3}{2}}\|\rho\dot{u}\|_{L^{2}}^{3}
  +C\int_{0}^{T}\sigma^{2}\|H\|_{L^{2}}^{\frac{1}{2}}\|\nabla H\|_{L^{2}}^{\frac{3}{2}}\|H\|_{L^{3}}^{3}\|\nabla^{2}H\|_{L^{2}}^{3}\non\\[2mm]
  \D&\leq C\bi(\sup_{0 \leq t\leq T}\|\nabla u\|_{L^{2}}^{2}\bi)^{\frac{1}{2}}
  \bi(\sup_{0 \leq t\leq T}\sigma^{2}\|\rho\dot{u}\|_{L^{2}}^{2}\bi)^{\frac{1}{2}}
  \int_{0}^{T}\sigma\|\rho\dot{u}\|_{L^{2}}^{2}\non\\[2mm]
  \D&\quad +C\bi(\sup_{0 \leq t\leq T}\|P\|_{L^{2}}\bi)
  \bi(\sup_{0 \leq t\leq T}\sigma^{2}\|\rho\dot{u}\|_{L^{2}}^{2}\bi)^{\frac{1}{2}}
  \int_{0}^{T}\sigma\|\rho\dot{u}\|_{L^{2}}^{2}\non\\[2mm]
  \D&\quad+C\bi(\sup_{0 \leq t\leq T}\|\nabla u\|_{L^{2}}^{2}\bi)^{\frac{1}{2}}\bi(\sup_{0 \leq t\leq T}\|H\|_{L^{2}}^{\frac{3}{2}}\|\nabla H\|_{L^{2}}^{\frac{3}{2}}\bi)
  \bi(\sup_{0 \leq t\leq T}\sigma^{2}\|\nabla^{2}H\|_{L^{2}}^{2}\bi)^{\frac{1}{2}}
  \int_{0}^{T}\sigma\|\nabla^{2}H\|_{L^{2}}^{2}\non\\[2mm]
  \D&\quad+C\bi(\sup_{0 \leq t\leq T}\|P\|_{L^{2}}\bi)
  \bi(\sup_{0 \leq t\leq T}\|H\|_{L^{2}}^{\frac{3}{2}}\|\nabla H\|_{L^{2}}^{\frac{3}{2}}\bi)
  \bi(\sup_{0 \leq t\leq T}\sigma^{2}\|\nabla^{2}H\|_{L^{2}}^{2}\bi)^{\frac{1}{2}}
  \int_{0}^{T}\sigma\|\nabla^{2}H\|_{L^{2}}^{2}\non\\[2mm]
  \D&\quad+C\bi(\sup_{0 \leq t\leq T}\|H\|_{L^{2}}^{2}\bi)^{\frac{1}{4}}
  \bi(\sup_{0 \leq t\leq T}\|\nabla H\|_{L^{2}}^{2}\bi)^{\frac{3}{4}}
  \bi(\sup_{0 \leq t\leq T}\sigma^{2}\|\rho^{\frac{1}{2}}\dot{u}\|_{L^{2}}^{2}\bi)^{\frac{1}{2}}
  \int_{0}^{T}\sigma\|\rho\dot{u}\|_{L^{2}}^{2}\non\\[2mm]
  \D&\quad+C\bi(\sup_{0 \leq t\leq T}\|H\|_{L^{2}}^{2}\bi)^{\frac{1}{4}}
  \bi(\sup_{0 \leq t\leq T}\|\nabla H\|_{L^{2}}^{2}\bi)^{\frac{3}{4}}
  \bi(\sup_{0 \leq t\leq T}\|H\|_{L^{2}}^{\frac{3}{2}}\|\nabla H\|_{L^{2}}^{\frac{3}{2}}\bi)\non\\[2mm]
  \D&\quad\quad\cdot
  \bi(\sup_{0 \leq t\leq T}\sigma^{2}\|\nabla^{2}H\|_{L^{2}}^{2}\bi)^{\frac{1}{2}}
  \int_{0}^{T}\sigma\|\nabla^{2}H\|_{L^{2}}^{2}\non\\[2mm]
  \D&\leq C(K_{6}^{\frac{1}{2}}+1)(K_{7}+1)\left((\gamma-1)^{\frac{1}{6}}E_{0}^{\frac{1}{2}}\right)^{\frac{1}{4}}
  (\gamma-1)^{\frac{2}{3}}E_{0}^{\frac{2}{3}}\non\\[2mm]
  \D&\quad+C(K_{7}+1)\left((\gamma-1)^{\frac{1}{6}}E_{0}^{\frac{1}{2}}\right)^{\frac{1}{4}}(\gamma-1)^{\frac{7}{6}}E_{0}^{\frac{7}{6}}\non\\[2mm]
  \D&\quad+CK_{2}^{\frac{9}{4}}M_{2}^{\frac{3}{4}}(K_{6}^{\frac{1}{2}}+1)\left((\gamma-1)^{\frac{1}{6}}E_{0}^{\frac{1}{2}}\right)^{\frac{3}{2}}
  +CK_{2}^{\frac{9}{4}}M_{2}^{\frac{3}{4}}(\gamma-1)^{\frac{1}{2}}E_{0}^{\frac{1}{2}}
  \left((\gamma-1)^{\frac{1}{6}}E_{0}^{\frac{1}{2}}\right)^{\frac{3}{2}}\non\\[2mm]
  \D&\quad+CK_{2}^{\frac{5}{4}}M_{2}^{\frac{3}{4}}(K_{7}+1)
  \left((\gamma-1)^{\frac{1}{6}}E_{0}^{\frac{1}{2}}\right)^{\frac{1}{2}}
  (\gamma-1)^{\frac{2}{3}}E_{0}^{\frac{2}{3}}\non\\[2mm]
  \D&\quad+CK_{2}^{\frac{7}{2}}M_{2}^{\frac{3}{2}}\left((\gamma-1)^{\frac{1}{6}}E_{0}^{\frac{1}{2}}\right)^{\frac{3}{2}} \non\\[2mm]
  \D&\leq K_{8}^{\tilde{1}}\left((\gamma-1)^{\frac{1}{6}}E_{0}^{\frac{1}{2}}\right)^{\frac{3}{2}},
 \end{align}
provided $(\gamma-1)^{\frac{1}{6}}E_{0}^{\frac{1}{2}}\leq 1$, where
 \begin{align}
   \D K_{8}^{\tilde{1}}&=C(K_{6}^{\frac{1}{2}}+1)(K_{7}+1)(\gamma-1)^{\frac{4}{9}}
   +C(K_{7}+1)(\gamma-1)^{\frac{7}{9}}+CK_{2}^{\frac{9}{4}}M_{2}^{\frac{3}{4}}(K_{6}^{\frac{1}{2}}+1)\non\\
   \D&\quad +CK_{2}^{\frac{9}{4}}M_{2}^{\frac{3}{4}}(\gamma-1)^{\frac{1}{3}}
   +CK_{2}^{\frac{5}{4}}M_{2}^{\frac{3}{4}}(K_{7}+1)(\gamma-1)^{\frac{4}{9}}
    +CK_{2}^{\frac{7}{2}}M_{2}^{\frac{3}{2}}.
 \end{align}
 Substituting (\ref{2dbu-E3.098}) and (\ref{2dbu-E3.099}) into (\ref{2dbu-E3.097}), assuming $((\gamma-1)^{\frac{1}{6}}E_{0}^{\frac{1}{2}}\leq 1$, we have
 \begin{align}\label{2dbu-E3.0101}
  \D J_{7}&=C\int_{0}^{T}\sigma^{2}\|P\|_{L^{4}}^{4}\non\\
    \D&\leq C(\bar{\rho})(\gamma-1)E_{0}+C\int_{0}^{T}\sigma^{2}\|G\|_{L^{4}}^{4}+\int_{0}^{T}\sigma^{2}\|H\|_{L^{8}}^{8}\non\\[2mm]
    \D&\leq K_{8}^{\tilde{2}}\left((\gamma-1)^{\frac{1}{6}}E_{0}^{\frac{1}{2}}\right)^{\frac{3}{2}},
  \end{align}
 where
 \begin{align}
   \D K_{8}^{\tilde{2}}=C(\bar{\rho})(\gamma-1)^{\frac{2}{3}}+CK_{2}^{\frac{15}{4}}M_{2}^{\frac{3}{4}}+K_{8}^{\tilde{1}}.
 \end{align}
 Combining (\ref{2dbu-E3.088})-(\ref{2dbu-E3.094}) and (\ref{2dbu-E3.0101}), we consequently get that
 \begin{align}\label{2dbu-E3.0103}
   &\D\mathscr{T}_{5}=C\int_{0}^{T}\intt\sigma^{2}|\nabla u|^{4}\leq K_{8}\left((\gamma-1)^{\frac{1}{6}}E_{0}^{\frac{1}{2}}\right)^{\frac{3}{2}},
 \end{align}
 where
 \begin{align}
  \D K_{8}&=C(\bar{\rho})(K_{6}^{\frac{1}{2}}+1)(K_{7}+1)(\gamma-1)^{\frac{4}{9}}
  +C(\bar{\rho})(K_{7}+1)(\gamma-1)^{\frac{7}{9}}\non\\[2mm]
  \D&\quad +CK_{2}^{\frac{1}{2}}(K_{7}+1)(\gamma-1)^{\frac{4}{9}}
  +CK_{2}^{\frac{9}{4}}M_{2}^{\frac{3}{4}}(K_{6}^{\frac{1}{2}}+1)\non\\[2mm]
  \D&\quad+CK_{2}^{\frac{9}{4}}M_{2}^{\frac{3}{4}}(\gamma-1)^{\frac{1}{3}}+CK_{2}^{3}M_{2}^{\frac{3}{4}}+K_{8}^{\tilde{2}}.
 \end{align}
  It holds from (\ref{2dbu-E3.028}) that
 \begin{align}\label{2dbu-E3.0105}
  \D \mathscr{T}_{1}=C\int_{0}^{\sigma(T)}\intt|\nabla u|^{2}\leq CK_{3}(\gamma-1)^{\frac{1}{6}}E_{0}^{\frac{1}{2}}.
 \end{align}
 To estimate $\mathscr{T}_{2}$ , due to $\mathrm{H\ddot{o}lder}$ inequality, (\ref{2dbu-E3.04}), (\ref{2dbu-E1.011}) and (\ref{2dbu-E3.0103}), we deduce
 \begin{align}
   \D \mathscr{T}_{2}&= C\int_{0}^{T}\sigma\intt|\nabla u|^{3}\non\\
   \D&\leq C\left(\int_{0}^{T}\intt |\nabla u|^{2}\right)^{\frac{1}{2}}
   \left(\int_{0}^{T}\intt \sigma^{2}|\nabla u|^{4}\right)^{\frac{1}{2}}\non\\
   \D&\leq CE_{0}^{\frac{1}{2}} K_{8}^{\frac{1}{2}}\left((\gamma-1)^{\frac{1}{6}}E_{0}^{\frac{1}{2}}\right)^{\frac{3}{4}}\non\\
   \D&\leq C K_{8}^{\frac{1}{2}}(\gamma-1)^{\frac{1}{12\times24}}\left((\gamma-1)^{\frac{1}{6}}E_{0}^{\frac{1}{2}}\right)^{\frac{7}{12}}.
 \end{align}
 As for $\mathscr{T}_{4}$, by $\mathrm{H\ddot{o}lder}$ inequality, (\ref{2dbu-E3.0101}) and (\ref{2dbu-E3.0103}), we get
 \begin{align}
  \D\mathscr{T}_{4}&= C\gamma^{2}\int_{0}^{T}\intt\sigma^{2}|P\nabla u|^{2}\non\\
  &\leq
  C\left(\int_{0}^{T}\sigma^{2}\|P\|_{L^{4}}^{4} \right)^{\frac{1}{2}}
  \left(\int_{0}^{T}\sigma^{2}\|\nabla u\|_{L^{4}}^{4} \right)^{\frac{1}{2}}\non\\
  &\leq C\sqrt{K_{8}^{\tilde{2}}}K_{8}^{\frac{1}{2}}
  \left((\gamma-1)^{\frac{1}{6}}E_{0}^{\frac{1}{2}}\right)^{\frac{3}{2}}.
 \end{align}
 Now it remains to estimate $\mathscr{T}_{3}$. It follows from (\ref{2dbu-E3.04}), (\ref{2dbu-E3.0101}) and (\ref{2dbu-E3.0103}) that
 \begin{align}\label{2dbu-E3.0108}
   \D\mathscr{T}_{3}&=C\gamma\int_{0}^{T}\sigma\intt P|\nabla u|^{2}\non\\
   \D&\leq C\left(\int_{0}^{T}\sigma^{2}\|P\|_{L^{4}}^{4} \right)^{\frac{1}{4}}
   \left(\int_{0}^{T}\sigma^{2}\|\nabla u\|_{L^{4}}^{4} \right)^{\frac{1}{4}}
   \left(\int_{0}^{T}\sigma^{2}\|\nabla u\|_{L^{2}}^{2} \right)^{\frac{1}{2}}\non\\
    \D&\leq CE_{0}^{\frac{1}{2}}\sqrt[4]{K_{8}^{\tilde{2}}K_{8}}\left((\gamma-1)^{\frac{1}{6}}E_{0}^{\frac{1}{2}}\right)^{\frac{3}{4}}\non\\
    \D&\leq C\sqrt[4]{K_{8}^{\tilde{2}}K_{8}}
    (\gamma-1)^{\frac{1}{12\times24}}\left((\gamma-1)^{\frac{1}{6}}E_{0}^{\frac{1}{2}}\right)^{\frac{7}{12}}.
 \end{align}
 Finally, we deduce from (\ref{2dbu-E3.0103}), (\ref{2dbu-E3.0105})-(\ref{2dbu-E3.0108}) that
 \begin{align}\label{2dbu-E3.0109}
  \D A_{1}(T)+A_{2}(T)&\leq C(K_{4}+K_{5})\left((\gamma-1)^{\frac{1}{6}}E_{0}^{\frac{1}{2}}\right)^{\frac{11}{18}}+
  CK_{3}(\gamma-1)^{\frac{1}{6}}E_{0}^{\frac{1}{2}}\non\\
  \D&\quad+C K_{8}^{\frac{1}{2}}(\gamma-1)^{\frac{1}{12\times24}}\left((\gamma-1)^{\frac{1}{6}}E_{0}^{\frac{1}{2}}\right)^{\frac{7}{12}}\non\\
  \D&\quad+C\sqrt[4]{K_{8}^{\tilde{2}}K_{8}}
    (\gamma-1)^{\frac{1}{12\times24}}\left((\gamma-1)^{\frac{1}{6}}E_{0}^{\frac{1}{2}}\right)^{\frac{7}{12}}\non\\
    \D&\quad+ C\sqrt{K_{8}^{\tilde{2}}}K_{8}^{\frac{1}{2}}
  \left((\gamma-1)^{\frac{1}{6}}E_{0}^{\frac{1}{2}}\right)^{\frac{3}{2}}
  +K_{8}\left((\gamma-1)^{\frac{1}{6}}E_{0}^{\frac{1}{2}}\right)^{\frac{3}{2}}\non\\
  \D&\leq CK_{9}\left((\gamma-1)^{\frac{1}{6}}E_{0}^{\frac{1}{2}}\right)^{\frac{11}{18}}\non\\
  \D&\leq  \left((\gamma-1)^{\frac{1}{6}}E_{0}^{\frac{1}{2}}\right)^{\frac{1}{2}},
 \end{align}
 provided that
 \begin{align}
   \D (\gamma-1)^{\frac{1}{6}}E_{0}^{\frac{1}{2}}\leq \min\left\{(CK_{9})^{-9},1\right\}.
 \end{align}
 where $K_{9}$ is given by
 \begin{align}
   \D K_{9}&=(K_{4}+K_{5})+K_{3}+ K_{8}^{\frac{1}{2}}(\gamma-1)^{\frac{1}{12\times24}}+\sqrt[4]{K_{8}^{\tilde{2}}K_{8}}
    (\gamma-1)^{\frac{1}{12\times24}}\non\\
    \D&\quad+\sqrt{K_{8}^{\tilde{2}}}K_{8}^{\frac{1}{2}}+K_{8}.
 \end{align}
 And to get (\ref{2dbu-E3.0109}), we have used the facts that $(\gamma-1)^{\frac{1}{6}}E_{0}^{\frac{1}{2}}\leq 1$. Then we finish the proof of
 Lemma \ref{2dbu-L3.9}.
 \end{proof}
 Now we are ready to prove the upper bound of the density.
 \begin{lem}\label{2dbu-L3.10}
 Under the same assumption as in Proposition \ref{2dbu-P3.1}, it holds that
 \begin{align}
   \D\sup_{0\leq t\leq T}\|\rho\|_{L^{\infty}}\leq \frac{7}{4}\bar{\rho},
 \end{align}
 provided that
 \begin{align}
   \D (\gamma-1)^{\frac{1}{6}}E_{0}^{\frac{1}{2}}\leq \min\left\{\varepsilon_{5},\left(\frac{\bar{\rho}}{2K_{9}}\right)^{16},\frac{\bar{\rho}}{4C(\bar{\rho})(1+K_{2}^{2})}\right\},
 \end{align}
 where
 \begin{align}
 \D K_{10}=K_{7}^{\frac{5}{16}}+K_{2}^{\frac{3}{2}}M_{2}^{\frac{1}{2}}+K_{2}^{\frac{1}{2}}M_{2}^{\frac{1}{2}}
  +K_{2}^{\frac{3}{8}}M_{2}^{\frac{3}{8}}K_{7}^{\frac{1}{4}}.
 \end{align}
 \end{lem}
 \begin{proof}
  Let $D_{t}\triangleq \partial_{t}+u\cdot\nabla $ denote the material derivative operator. Then, in terms of the effective viscous flux $G$, we can rewrite (1.1) as
 $$
 \D D_{t}\rho=g(\rho)+b'(\rho),
 $$
 where
 \begin{align}
   \D g(\rho)\triangleq \frac{\rho P}{2\mu+\lambda},\ \ b(t)=-\frac{1}{2\mu+\lambda}\int_{0}^{t}\left(\rho G+\frac{1}{2}\rho|H|^{2}\right).
 \end{align}
 Moreover, it follows from Lemmas \ref{2dbu-L2.1}-\ref{2dbu-L2.2}, (\ref{2dbu-E2.04}) and (\ref{2dbu-E2.05}) that
 \begin{align}\label{2dbu-E3.115}
  \D \|G\|_{L^{\infty}}&\leq C\|G\|_{L^{6}}^{\frac{1}{2}}\|\nabla G\|_{L^{6}}^{\frac{1}{2}}\non\\[2mm]
  \D&\leq C\left(\|\rho\dot{u}\|_{L^{2}}^{\frac{1}{2}}+\|\nabla H\|_{L^{2}}^{\frac{3}{4}}\|\nabla^{2} H\|_{L^{2}}^{\frac{1}{4}}\right)
  \left(\|\nabla\dot{u}\|_{L^{2}}^{\frac{1}{2}}+\|\nabla H\|_{L^{2}}^{\frac{1}{4}}\|\nabla^{2} H\|_{L^{2}}^{\frac{3}{4}}\right)\non\\[2mm]
  \D&\leq C\|\rho\dot{u}\|_{L^{2}}^{\frac{1}{2}}\|\nabla\dot{u}\|_{L^{2}}^{\frac{1}{2}}
  +C\|\rho\dot{u}\|_{L^{2}}^{\frac{1}{2}}\|\nabla H\|_{L^{2}}^{\frac{1}{4}}\|\nabla^{2} H\|_{L^{2}}^{\frac{3}{4}}\non\\[2mm]
  \D&\quad+C\|\nabla\dot{u}\|_{L^{2}}^{\frac{1}{2}}\|\nabla H\|_{L^{2}}^{\frac{3}{4}}\|\nabla^{2} H\|_{L^{2}}^{\frac{1}{4}}
  +C\|\nabla H\|_{L^{2}}\|\nabla^{2} H\|_{L^{2}}.
 \end{align}
 For $t\in [0.\sigma(T)]$, one can deduce that for all $0\leq t_{1}\leq t_{2}\leq \sigma(T)$,
 \begin{align}\label{2dbu-E3.01117}
  \D|b(t_{2})-b(t_{1})|&\leq C\int_{0}^{\sigma(T)}\left(\|G\|_{L^{\infty}}+\|H\|_{L^{\infty}}^{2}\right) \non\\[2mm]
  \D&\leq C\int_{0}^{\sigma(T)}\|\rho\dot{u}\|_{L^{2}}^{\frac{1}{2}}\|\nabla\dot{u}\|_{L^{2}}^{\frac{1}{2}}
    +C\int_{0}^{\sigma(T)}\|\nabla H\|_{L^{2}}\|\nabla^{2} H\|_{L^{2}}\non\\[2mm]
    \D&\quad+\int_{0}^{\sigma(T)}\|\rho\dot{u}\|_{L^{2}}^{\frac{1}{2}}\|\nabla H\|_{L^{2}}^{\frac{1}{4}}\|\nabla^{2} H\|_{L^{2}}^{\frac{3}{4}}
    +\int_{0}^{\sigma(T)}\|\nabla\dot{u}\|_{L^{2}}^{\frac{1}{2}}\|\nabla H\|_{L^{2}}^{\frac{3}{4}}\|\nabla^{2} H\|_{L^{2}}^{\frac{1}{4}}\non\\[2mm]
  \D&\leq C\left(\int_{0}^{\sigma(T)}\|\rho\dot{u}\|_{L^{2}}^{\frac{2}{3}}t^{-\frac{1}{3}}\right)^{\frac{3}{4}}
   \left(\int_{0}^{\sigma(T)}t\|\nabla\dot{u}\|_{L^{2}}^{2}\right)^{\frac{1}{4}}
   \non\\[2mm]
   \D&\quad+C\left(\int_{0}^{\sigma(T)}\|\nabla H\|_{L^{2}}^{2}\right)^{\frac{1}{2}}
   \left(\int_{0}^{\sigma(T)}\|\nabla^{2} H\|_{L^{2}}^{2}\right)^{\frac{1}{2}}\non\\[2mm]
   \D&\quad+\left(\sup_{0\leq t\leq\sigma(T)}t^{2}\|\rho\dot{u}\|_{L^{2}}^{2}\right)^{\frac{1}{4}}
   \left(\sup_{0\leq t\leq\sigma(T)}\|\nabla H\|_{L^{2}}^{2}\right)^{\frac{1}{8}}
   \int_{0}^{\sigma(T)}t^{-\frac{1}{2}}\|\nabla^{2} H\|_{L^{2}}^{\frac{3}{4}}\non\\[2mm]
   \D&\quad+C\left(\sup_{0\leq t\leq\sigma(T)}\|\nabla H\|_{L^{2}}^{2}\right)^{\frac{3}{8}}
   \left(\sup_{0\leq t\leq\sigma(T)}t^{2}\|\nabla^{2} H\|_{L^{2}}^{2}\right)^{\frac{1}{8}}
   \int_{0}^{\sigma(T)}\left(t^{\frac{1}{4}}\|\nabla\dot{u}\|_{L^{2}}^{\frac{1}{2}}\right)t^{-\frac{1}{2}}\non\\[2mm]
   \D&\leq CK_{7}^{\frac{1}{4}}\left(\int_{0}^{\sigma(T)}\bi(\|\rho\dot{u}\|_{L^{2}}^{2}t\bi)^{\frac{1}{3}}t^{-\frac{2}{3}}\right)^{\frac{3}{4}}
   +CK_{2}^{\frac{3}{2}}M_{2}^{\frac{1}{2}}\left((\gamma-1)^{\frac{1}{6}}E_{0}^{\frac{1}{2}}\right)^{\frac{1}{2}}\non\\[2mm]
   \D&\quad+CK_{2}^{\frac{1}{8}}M_{2}^{\frac{1}{8}}\left((\gamma-1)^{\frac{1}{6}}E_{0}^{\frac{1}{2}}\right)^{\frac{1}{8}}
    \left(\int_{0}^{\sigma(T)}\|\nabla^{2} H\|_{L^{2}}^{2}\right)^{\frac{3}{8}}
   \left(\int_{0}^{\sigma(T)}t^{-\frac{4}{5}}\right)^{\frac{5}{8}}\non\\[2mm]
   \D&\quad+CK_{2}^{\frac{3}{8}}M_{2}^{\frac{3}{8}}\left((\gamma-1)^{\frac{1}{6}}E_{0}^{\frac{1}{2}}\right)^{\frac{1}{16}}
   \left(\int_{0}^{\sigma(T)}t\|\nabla\dot{u}\|_{L^{2}}^{2}\right)^{\frac{1}{4}}
   \left(\int_{0}^{\sigma(T)}t^{-\frac{2}{3}}\right)^{\frac{3}{4}}\non\\[2mm]
   \D&\leq CK_{7}^{\frac{1}{4}}\bi(\sup_{0 \leq t\leq\sigma(T)}\bi(\|\rho\dot{u}\|_{L^{2}}^{2}t\bi)^{\frac{1}{16}}\bi)
   \left(\int_{0}^{\sigma(T)}\bi(\|\rho\dot{u}\|_{L^{2}}^{2}t\bi)^{\frac{1}{4}}t^{-\frac{2}{3}}\right)^{\frac{3}{4}}\non\\[2mm]
   \D&\quad+CK_{2}^{\frac{3}{2}}M_{2}^{\frac{1}{2}}\left((\gamma-1)^{\frac{1}{6}}E_{0}^{\frac{1}{2}}\right)^{\frac{1}{2}}
   +CK_{2}^{\frac{1}{2}}M_{2}^{\frac{1}{2}}\left((\gamma-1)^{\frac{1}{6}}E_{0}^{\frac{1}{2}}\right)^{\frac{1}{8}}\non\\[2mm]
   \D&\quad+CK_{2}^{\frac{3}{8}}M_{2}^{\frac{3}{8}}
   K_{7}^{\frac{1}{4}}\left((\gamma-1)^{\frac{1}{6}}E_{0}^{\frac{1}{2}}\right)^{\frac{1}{16}}\non\\[2mm]
   \D&\leq CK_{7}^{\frac{5}{16}}\left(\int_{0}^{\sigma(T)}t\|\rho\dot{u}\|_{L^{2}}^{2}\right)^{\frac{3}{16}}
   \left(\int_{0}^{\sigma(T)}t^{-\frac{8}{9}}\right)^{\frac{9}{16}}
   +CK_{2}^{\frac{3}{2}}M_{2}^{\frac{1}{2}}\left((\gamma-1)^{\frac{1}{6}}E_{0}^{\frac{1}{2}}\right)^{\frac{1}{2}}\non\\[2mm]
   \D&\quad+CK_{2}^{\frac{1}{2}}M_{2}^{\frac{1}{2}}\left((\gamma-1)^{\frac{1}{6}}E_{0}^{\frac{1}{2}}\right)^{\frac{1}{8}}
   +CK_{2}^{\frac{3}{8}}M_{2}^{\frac{3}{8}}K_{7}^{\frac{1}{4}}
   \left((\gamma-1)^{\frac{1}{6}}E_{0}^{\frac{1}{2}}\right)^{\frac{1}{16}}\non\\[2mm]
   \D&\leq CK_{7}^{\frac{5}{16}}A_{1}(\sigma(T))^{\frac{3}{16}}
  +CK_{2}^{\frac{3}{2}}M_{2}^{\frac{1}{2}}\left((\gamma-1)^{\frac{1}{6}}E_{0}^{\frac{1}{2}}\right)^{\frac{1}{2}}
  +CK_{2}^{\frac{1}{2}}M_{2}^{\frac{1}{2}}\left((\gamma-1)^{\frac{1}{6}}E_{0}^{\frac{1}{2}}\right)^{\frac{1}{8}}\non\\[2mm]
 \D&\quad+CK_{2}^{\frac{3}{8}}M_{2}^{\frac{3}{8}}K_{7}^{\frac{1}{4}}
 \left((\gamma-1)^{\frac{1}{6}}E_{0}^{\frac{1}{2}}\right)^{\frac{1}{16}}\non\\[2mm]
  \D&\leq C\Big(K_{7}^{\frac{5}{16}}+K_{2}^{\frac{3}{2}}M_{2}^{\frac{1}{2}}+K_{2}^{\frac{1}{2}}M_{2}^{\frac{1}{2}}
  +K_{2}^{\frac{3}{8}}M_{2}^{\frac{3}{8}}K_{7}^{\frac{1}{4}}\Big)
  \left((\gamma-1)^{\frac{1}{6}}E_{0}^{\frac{1}{2}}\right)^{\frac{1}{16}}\non\\[2mm]
  \D&\triangleq K_{10}\left((\gamma-1)^{\frac{1}{6}}E_{0}^{\frac{1}{2}}\right)^{\frac{1}{16}},
 \end{align}
 where (\ref{2dbu-E3.9}), (\ref{2dbu-E3.0062}), (\ref{2dbu-E3.85}) and (\ref{2dbu-E3.115}) have been used.
 Therefore, for $t\in [0,\sigma(T)]$, we can choose $N_{0}$ and $N_{1}$ in Lemma 2.3 as follows
 \begin{align}
   \D N_{1}=0,\ \ N_{0}=K_{10}\left((\gamma-1)^{\frac{1}{6}}E_{0}^{\frac{1}{2}}\right)^{\frac{1}{16}},
 \end{align}
 and $\bar{\zeta}=0$. Then
 \begin{align}
   \D g(\zeta)=-\frac{\zeta P}{2\mu+\lambda}\leq -N_{1}=0\  \mathrm{for\  all}\  \zeta\geq \bar{\zeta}=0.
 \end{align}
 We thus have
 \begin{align}
  \D \sup_{0\leq t\leq \sigma(T)}\|\rho\|_{L^{\infty}}\leq\max\{\bar{\rho},0\}+N_{0}\leq \bar{\rho}+K_{10}\left((\gamma-1)^{\frac{1}{6}}E_{0}^{\frac{1}{2}}\right)^{\frac{1}{16}}\leq \frac{3\bar{\rho}}{2},
 \end{align}
 provided
 \begin{align}
  \D (\gamma-1)^{\frac{1}{6}}E_{0}^{\frac{1}{2}}\leq \min\left\{\left(\frac{\bar{\rho}}{2K_{9}}\right)^{16},\varepsilon_{5}\right\}.
 \end{align}
 Furthermore, due to Lemmas \ref{2dbu-L2.1}-\ref{2dbu-L2.2}, for $t\in [\sigma(T), T]$, we can derive
 \begin{align}\label{2dbu-E3.122}
  \D|b(t_{2})-b(t_{1})|&\leq C\int_{t_{1}}^{t_{2}}\left(\|G\|_{L^{\infty}}+\|H\|_{L^{\infty}}^{2}\right) \non\\[2mm]
  \D&\leq \frac{C(\bar{\rho})}{2\mu+\lambda}(t_{2}-t_{1})+C(\bar{\rho})\int_{\sigma(T)}^{T}\|G\|_{L^{\infty}}^{4}+C\int_{\sigma(T)}^{T}\|\nabla H\|_{L^{2}}\|\nabla^{2} H\|_{L^{2}}\non\\[2mm]
   \D&\leq \frac{C(\bar{\rho})}{2\mu+\lambda}(t_{2}-t_{1})+
   +C\left(\int_{\sigma(T)}^{T}\|\nabla H\|_{L^{2}}^{2}\right)^{\frac{1}{2}}
   \left(\int_{\sigma(T)}^{T}\|\nabla^{2} H\|_{L^{2}}^{2}\right)^{\frac{1}{2}}\non\\[2mm]
   \D&\quad+C(\bar{\rho})\int_{\sigma(T)}^{T}\|\rho\dot{u}\|_{L^{2}}^{2}\|\nabla\dot{u}\|_{L^{2}}^{2}
   +C\int_{\sigma(T)}^{T}\|\rho\dot{u}\|_{L^{2}}^{2}\|\nabla H\|_{L^{2}}\|\nabla^{2} H\|_{L^{2}}^{3}\non\\[2mm]
   \D&\quad+C\int_{\sigma(T)}^{T}\|\nabla\dot{u}\|_{L^{2}}^{2}\|\nabla H\|_{L^{2}}^{3}\|\nabla^{2} H\|_{L^{2}}
   +C\int_{\sigma(T)}^{T}\|\nabla H\|_{L^{2}}^{2}\|\nabla^{2} H\|_{L^{2}}^{2}
   \non\\[2mm]
   \D&\leq \frac{C(\bar{\rho})}{2\mu+\lambda}(t_{2}-t_{1})+C(\bar{\rho})K_{2}^{2}(\gamma-1)^{\frac{1}{6}}E_{0}^{\frac{1}{2}}
   +C(\bar{\rho})A_{2}(T)\int_{\sigma(T)}^{T}\|\nabla\dot{u}\|_{L^{2}}^{2}
   \non\\[2mm]
   \D&\quad+CA_{2}(T)^{\frac{3}{2}}A_{1}(T)^{\frac{3}{2}}
   +A_{1}(T)^{\frac{3}{2}}A_{2}(T)^{\frac{1}{2}}\int_{\sigma(T)}^{T}\|\nabla\dot{u}\|_{L^{2}}^{2}
   +CA_{1}(T)^{2}\non\\[2mm]
      \D&\leq \frac{C(\bar{\rho})}{2\mu+\lambda}(t_{2}-t_{1})
      +C(\bar{\rho})(\gamma-1)^{\frac{1}{6}}E_{0}^{\frac{1}{2}}+C(\bar{\rho})K_{2}^{2}A_{2}(T)^{2}\non\\[2mm]
      \D&\quad+CA_{2}(T)^{\frac{3}{2}}A_{1}(T)^{\frac{3}{2}}+CA_{1}(T)^{2}\non\\[2mm]
   \D&\leq \frac{C(\bar{\rho})}{2\mu+\lambda}(t_{2}-t_{1})+C(\bar{\rho})(1+K_{2}^{2})(\gamma-1)^{\frac{1}{6}}E_{0}^{\frac{1}{2}}.
  \end{align}
 Consequently, for $t\in [\sigma(T), T]$, we can choose $N_{0}$ and $N_{1}$ in Lemma 2.3 as follows
 \begin{align}
   \D N_{1}=\frac{1}{2\mu+\lambda},\ \ N_{0}=C(\bar{\rho})(1+K_{2}^{2})(\gamma-1)^{\frac{1}{6}}E_{0}^{\frac{1}{2}}.
 \end{align}
 Noticing that
 \begin{align}
   \D g(\zeta)=-\frac{\zeta P(\zeta)}{2\mu+\lambda}\leq-N_{1}=-\frac{1}{2\mu+\lambda}\ \mathrm{for}\ \mathrm{all}\ \zeta\geq 1,
 \end{align}
 one can set $\bar{\zeta}=1$. Thus
   \begin{align}
  \D \sup_{\sigma(T)\leq t\leq T}\|\rho\|_{L^{\infty}}\leq\max\{\frac{3}{2}\bar{\rho},1\}+N_{0}\leq \frac{3}{2}\bar{\rho}+C(\bar{\rho})(1+K_{2}^{2})(\gamma-1)^{\frac{1}{6}}E_{0}^{\frac{1}{2}}\leq \frac{7\bar{\rho}}{4},
 \end{align}
provided
\begin{align}
  \D (\gamma-1)^{\frac{1}{6}}E_{0}^{\frac{1}{2}}\leq  \min\left\{\varepsilon_{5},\frac{\bar{\rho}}{4C(\bar{\rho})(1+K_{2}^{2})}\right\}.
\end{align}
 \end{proof}
\section{Proof of Theorem \ref{2dbu-T1.1}}
\qquad In this section, we devote to prove the main result of this paper. First, from now on we always assume that the conditions in Theorem \ref{2dbu-T1.1} hold. Moreover, we denote the generic constant by $C$ which may depends on $T$, $\mu$, $\lambda$, $\nu$, $\gamma$, $\bar{\rho}$,  $\tilde{\rho}$, $M_{1}$,$M_{2}$, $g$ and some other initial data. Here $g\in L^{2}$ is the function in the compatibility condition (\ref{2dbu-E1.8}).
The following higher-order a priori estimates of the smooth solutions which are needed to guarantee the classical solutions  $(\rho, u, H)$ to be global ones have been proved in \cite{H.L. L}, so we omit their proof here.
\begin{lem}
The following estimates hold:
\begin{align}
  \D \sup_{0\leq t\leq T}&(\|\nabla H\|_{L^{2}}^{2}+\|\nabla u\|_{L^{2}}^{2})\non\\
  \D&+\int_{0}^{T}\Big(\|\rho^{\frac{1}{2}}\dot{u}\|_{L^{2}}^{2}+\|H_{t}\|_{L^{2}}^{2}+\|\nabla^{2}H\|_{L^{2}}^{2}\Big)\leq C(T),\\
  \D\sup_{0\leq t\leq T}&(\|\rho^{\frac{1}{2}}\dot{u}\|_{L^{2}}^{2}+\|\nabla^{2}H\|_{L^{2}}^{2}+\|H_{t}\|_{L^{2}}^{2})\non\\
  \D&+\int_{0}^{T}\Big(\|\nabla\dot{u}\|_{L^{2}}^{2}+\|\nabla H_{t}\|_{L^{2}}^{2}\Big)\leq C(T),\\
  \D\sup_{0\leq t\leq T}&(\|\nabla\rho\|_{L^{2}\cap L^{6}}+\|\nabla u\|_{H^{1}})+\int_{0}^{T}\|\nabla u\|_{L^{\infty}}\leq C(T),\\
  \D\sup_{0\leq t\leq T}&(\|\rho^{\frac{1}{2}}u_{t}\|_{L^{2}}^{2})+\int_{0}^{T}\|\nabla u_{t}\|_{L^{2}}^{2} \leq C(T),\\
  \D\sup_{0\leq t\leq T}&(\|\nabla \rho\|_{H^{1}}+\|\nabla P\|_{H^{1}})+\int_{0}^{T}\|\nabla^{2}u\|_{H^{1}}^{2} \leq C(T),\\
  \D\sup_{0\leq t\leq T}&(\|\rho_{t}\|_{H^{1}}+\|P_{t}\|_{H^{1}})
  +\int_{0}^{T}(\|\rho_{tt}\|_{L^{2}}^{2}+\|P_{tt}\|_{L^{2}}^{2})\leq C(T),\\
  \D\sup_{0\leq t\leq T}&\sigma(\|\nabla u\|_{H^{2}}^{2}+\|\nabla u_{t}\|_{L^{2}}^{2}
  +\|\nabla H_{t}\|_{L^{2}}^{2}+\|\nabla H\|_{H^{2}}^{2})\non\\
  \D&+\int_{0}^{T}\sigma\Big(\|\rho^{\frac{1}{2}}u_{tt}\|_{L^{2}}^{2}
  +\|\nabla u_{t}\|_{H^{1}}^{2}+\|H_{tt}\|_{L^{2}}^{2}\Big)\leq C(T),\\
  \D\sup_{0\leq t\leq T}&(\|\nabla \rho\|_{W^{1,q}}+\|\nabla P\|_{W^{1,q}})\non\\
  \D&+\int_{0}^{T}\Big(\|\nabla u_{t}\|_{L^{q}}^{p_{0}}+\|\nabla^{2} u\|_{W^{1,q}}^{p_{0}}\Big)\leq C(T),
 \end{align}
 for fixed $q\in (3,6)$, where $ 1\leq p_{0}<\frac{4q}{5q-6}\in (1,2)$.
 \begin{align}
 \D\sup_{0\leq t\leq T}&\sigma(\|\rho^{\frac{1}{2}}u_{tt}\|_{L^{2}}+\|\nabla^{2}u_{t}\|_{L^{2}}+\|\nabla^{2}u\|_{W^{1,q}}+\|\nabla^{2}H\|_{H^{2}}
 +\|\nabla^{2}H_{t}\|_{L^{2}}+\|H_{tt}\|_{L^{2}})\non\\
 \D&+\int_{0}^{T}\sigma^{2}\Big(\|\nabla u_{tt}\|_{L^{2}}^{2}+\|\nabla H_{tt}\|_{L^{2}}^{2}\Big)\leq C(T).
 \end{align}
\end{lem}
Thanks to all the a priori estimates established above, we now are ready to prove Theorem \ref{2dbu-T1.1}. In fact, this can be done in a method the same as that in \cite{H.L. L}, we omit it here for simplicity.

\section*{Acknowledgement}
This work was
supported  by the National Natural Science Foundation of China
$\#$11331005, the Program for Changjiang Scholars and Innovative
Research Team in University $\#$IRT13066, and the Special Fund for
Basic Scientific Research of Central Colleges $\#$CCNU12C01001.
The first author was also supported by excellent doctorial
dissertation cultivation grant from Central Normal University.

\end{document}